\documentclass[11pt,a4paper]{amsart}
\usepackage{graphicx}
\usepackage[left=25mm,top=25mm,bottom=25mm,right=25mm]{geometry}
\usepackage{amssymb,dsfont,amsfonts}
\usepackage{mathtools}
\usepackage[utf8]{inputenc} 
\usepackage{cmap}           
\usepackage{hyperxmp}       
\usepackage[english]{babel} 
\usepackage[pdfdisplaydoctitle = true,
colorlinks = true,
urlcolor = blue,
citecolor = blue,
linkcolor = blue,
pdfstartview = FitH,
pdfpagemode = UseNone,
bookmarksnumbered = true,
unicode=true]{hyperref} 
\usepackage{algpseudocode}
\usepackage{booktabs}
\graphicspath{{figs/}}
\newtheorem{thm}{Theorem}[section]

\newtheorem{lem}[thm]{Lemma}

\theoremstyle{definition}
\newtheorem{defn}[thm]{Definition}
\theoremstyle{remark}
\newtheorem{rem}[thm]{Remark}
\newtheorem{exmp}[thm]{Example}

\DeclarePairedDelimiter\ceil{\lceil}{\rceil}

\newcommand{\RR}{\mathbb{R}}                
\newcommand{\ZZ}{\mathbb{Z}}                
\newcommand{\NN}{\mathbb{N}}                

\newcommand{\Fix}{\mathbb{F}\mathrm{ix}}    
\newcommand{\Ela}{\mathbb{E}\mathrm{la}}    
\newcommand{\Piez}{\mathbb{P}\mathrm{iez}}  
\newcommand{\HH}{\mathbb{H}}                
\newcommand{\Sym}{\mathbb{S}}               
\newcommand{\VV}{\mathbb{V}}                

\newcommand{\GL}{\mathrm{GL}}               
\newcommand{\OO}{\mathrm{O}}                
\newcommand{\SO}{\mathrm{SO}}               
\newcommand{\octa}{\mathbb{O}}              
\newcommand{\DD}{\mathbb{D}}                
\newcommand{\triv}{\mathds{1}}		        

\newcommand{\vv}{\pmb{v}}                   
\newcommand{\ww}{\pmb{w}}                   
\newcommand{\xx}{\pmb{x}}                   
\newcommand{\yy}{\pmb{y}}                   

\newcommand{\re}{\mathrm{e}}

\newcommand{\ba}{\mathbf{a}}
\newcommand{\bb}{\mathbf{b}}

\newcommand{\bd}{\mathbf{d}}
\newcommand{\be}{\mathbf{e}}
\newcommand{\bq}{\mathbf{q}}                
\newcommand{\bg}{\mathbf{g}}
\newcommand{\bh}{\mathbf{h}}
\newcommand{\bp}{\mathbf{p}}

\newcommand{\bv}{\mathbf{v}}

\newcommand{\bE}{\mathbf{E}}                
\newcommand{\bH}{\mathbf{H}}                
\newcommand{\bA}{\mathbf{A}}                
\newcommand{\bT}{\mathbf{T}}                

\newcommand{\lc}{\pmb \varepsilon}          

\DeclareMathOperator{\tr}{tr}

\DeclareMathOperator{\Orb}{Orb}
\DeclareMathOperator{\grad}{grad}
\DeclareMathOperator{\3dots}{\raisebox{-0.25ex}{\vdots}}

\newcommand{\norm}[1]{\lVert#1\rVert}                   
\newcommand{\abs}[1]{\lvert#1\rvert}                    
\newcommand{\set}[1]{\left\{#1\right\}}                 
\newcommand{\strata}[1]{\Sigma_{[#1]}}	                
\newcommand{\cstrata}[1]{\overline{\Sigma}_{[#1]}}	    

\hypersetup{
	pdfauthor={P. Azzi, R. Desmorat, B. Kolev & F. Priziac}, 
	pdftitle={Distance to a constitutive tensor isotropy stratum by Lasserre polynomial optimization method}, 
	pdfsubject={MSC 2020: 90C23; 14P10; 90C22; 74B05; 74E10}, 
	pdfkeywords={polynomial optimization; Lasserre's method; semidefinite programming; distance to a symmetry class; cubic symmetry; elasticity; piezoelectricity; semialgebraic and real algebraic geometry}, 
	pdflang=en, 
}
\begin{document}
\title[Distance to a constitutive tensor isotropy stratum]{Distance to a constitutive tensor isotropy stratum \protect\\
  by Lasserre polynomial optimization method}%

\author{P. Azzi}
\address[Perla Azzi]{CNRS, Sorbonne université, IMJ - Institut de mathématiques Jussieu, 75005, Paris, France \& Université Paris-Saclay, CentraleSupélec, ENS Paris-Saclay, CNRS, LMPS - Laboratoire de Mécanique Paris-Saclay, 91190, Gif-sur-Yvette, France}
\email[Perla Azzi]{perla.azzi@ens-paris-saclay.fr}

\author{R. Desmorat}
\address[Rodrigue Desmorat]{Université Paris-Saclay, CentraleSupélec, ENS Paris-Saclay, CNRS, Laboratoire de Mécanique Paris-Saclay, 91190, Gif-sur-Yvette, France.}
\email{rodrigue.desmorat@ens-paris-saclay.fr}

\author{B. Kolev}
\address[Boris Kolev]{Université Paris-Saclay, CentraleSupélec, ENS Paris-Saclay, CNRS, Laboratoire de Mécanique Paris-Saclay, 91190, Gif-sur-Yvette, France.}
\email{boris.kolev@ens-paris-saclay.fr}

\author{F. Priziac}
\address[Fabien Priziac]{Univ Bretagne Sud, CNRS UMR 6205, LMBA, F-56000 Vannes, France}
\email{fabien.priziac@univ-ubs.fr}

\date{July 9, 2022}

\thanks{The authors were partially supported by CNRS Projet 80--Prime GAMM (Géométrie algébrique complexe/réelle et mécanique des matériaux).}


\begin{abstract}
  We give a detailed description of a polynomial optimization method allowing to solve a problem in continuum mechanics: the determination of the elasticity or the piezoelectricity tensor of a specific isotropy stratum the closest to a given experimental tensor, and the calculation of the distance to the given tensor from the considered isotropy stratum. We take advantage of the fact that the isotropy strata are semialgebraic sets to show that the method, developed by Lasserre and coworkers which consists in solving polynomial optimization problems with semialgebraic constraints, successfully applies.
\end{abstract}

\maketitle


\begin{scriptsize}
  \setcounter{tocdepth}{2}
  \tableofcontents
\end{scriptsize}

\section{Introduction}
\label{sec:intro}

In mechanics, \emph{linear constitutive laws} are described by the orbit space of a representation of the three-dimensional orthogonal group on the vector space of the considered \emph{constitutive tensors}~$\bT$~\cite{LC1985,Cia1988,EM1990}. This orbit space is endowed with a natural stratification by \emph{isotropy classes} $[H]$, the strata $\strata{H}$ being the set of tensors with symmetry group conjugate to $H$.

The symmetry group of a measured (raw) tensor $\bT_{0}$ is in general trivial. However, in practice, appealing to Curie principle---\emph{the symmetries of the causes are to be found in the effects}---a symmetry of a constitutive tensor is often expected by observing the micro-structure of a material~\cite{Art1993,FBG1996,FGB1998}. For instance, the elasticity tensor of a single crystal alloy with cubic crystal network is expected to be cubic ($[\octa]$, see figure~\ref{fig:microstructure}), the piezoelectric tensor of an aluminum nitride (AlN) alloyed with rocksalt transition metal nitrides is expected to become cubic ($[\octa^{-}]$) for a high chromium concentration~\cite{MTG2018}. The mechanical problem thus comes down to the computation of the distance $d(\bT_{0}, \cstrata{H})$ of a raw constitutive tensor $\bT_{0}$ to a closed isotropy stratum $\cstrata{H}$.

In linear elasticity, which involves a fourth-order tensor $\bE$, the distance to an isotropy stratum has been formulated as the minimization problem~\cite{GTT1963,FGB1998,Del2005,MN2006,BS2008}
\begin{equation*}
  \min_{\bE\in \cstrata{H}} \norm{\bE_{0}-\bE},
  \qquad
  \bE=\rho_{4}(g) \bA,
  \qquad  \bA\in \Fix(H), \, g\in \SO(3),
\end{equation*}
with the natural parameterization by \emph{normal form} $\bA$ (a fixed point set for a representative symmetry group $H$) and \emph{rotation} $g$. This problem has, however, many local minima and several global minima, making the determination of all the solutions numerically difficult.

In this paper, we formulate the computation of the distance to an isotropy stratum as a polynomial optimization problem. To do so, we make use of the property that the closed isotropy strata $\cstrata{H}$ are basic closed semialgebraic sets~\cite{AS1981,AS1983,PS1985,Sch1989}. For the fourth-order elasticity tensor, an explicit characterization of the closed strata $\cstrata{H}$ by polynomial equations and inequalities has recently been obtained, by means of polynomial covariants~\cite[Theorem 10.2]{OKDD2021}. Since such a result is not yet available in piezoelectricity, we have provided in theorem~\ref{thm:Pcubic} a polynomial characterization of the cubic symmetry stratum $\cstrata{\octa^{-}}$ for the third-order piezoelectricity tensor.

We formulate the distance problems in question in such a way that we can apply a semialgebraic optimization method designed by Lasserre and coworkers~\cite{Las2001,Las2009,Las2015,HLL2009,JLL2014} to compute directly the global miminum of a polynomial function over polynomial constraints describing a basic closed semialgebraic set. This method consists in building a sequence of semidefinite programs whose optimal values converge to the desired minimum, under some hypothesis on the constraints. The benefit is that there exist efficient algorithms to solve numerically semidefinite programs, based on methods used in linear programming~\cite{Dan1963,LY2008,Van2020}, such as the ellipsoid method~\cite{GLS1981} or the interior point method~\cite{NN1994,Ali1995,NT1997,Tod2001,Jar1993,BCM2007,BCR2013,DHer2012}. The considered algorithm has been implemented by Lasserre and Henrion~\cite{HLL2009} in a Matlab freeware, named GloptiPoly~\cite{HL2003}, that aims to solve a sequence of relaxed semidefinite programs using SeDuMi (a Matlab toolbox for solving semidefinite programs created by Sturm~\cite{Stu1997,Stu1999}). Moreover, this algorithm, when its stopping criterion is satisfied, extracts (approximated) minimizers for the considered minimized function.

\subsection*{Organization of the paper}

The paper is organized as follows. In \autoref{sec:sym_classes}, we recall basic material on isotropy classes and we pose the problem of the distance to an isotropy class. In \autoref{sec:poly-opt}, we introduce polynomial optimization and its formulation to semidefinite programs. In  \autoref{sec:poly-opt} and \autoref{sec:algo}, we describe the Lasserre and coworkers method for solving polynomial optimization problems with semialgebraic constraints as well as the corresponding algorithm, implemented as the software GloptiPoly. As a direct application, we deal with three examples of constitutive tensors. In \autoref{sec:isoT}, we illustrate the method with the academic example of the distance of a symmetric second-order tensor to the transversely isotropic stratum $\cstrata{\OO(2)}$. In \autoref{sec:ela}, we compute the distance of an experimental elasticity (fourth-order) tensor of a Nickel-based single crystal superalloy to the cubic stratum $\cstrata{\octa}$, and consequently we extract the cubic elasticity tensor the closest to the experimental one. Finally, in \autoref{sec:piezo}, we detail how polynomial optimization allows to find the cubic piezoelectricity (third-order) tensors, in $\cstrata{\octa^{-}}$, the closest to raw tensors for wurtzite alloys.

All the tensorial components will be expressed with respect to an orthonormal basis. Hence, no distinction will be made between covariant and contravariant components. The notation $\bq=(\delta_{ij})$ stands for the Euclidean metric tensor.

\section{Isotropy classes and strata -- Distance to an isotropy stratum}
\label{sec:sym_classes}

Let $G$ be a compact group and $\rho : G \to \GL(\VV)$ be a continuous representation of $G$ on a finite dimensional real vector space $\VV$. Given $\vv \in \VV$, its \emph{orbit} is the subset of $\VV$ defined by
\begin{equation*}
  \Orb(\vv) := \set{\rho(g)\vv;\; g\in G},
\end{equation*}
and its \emph{symmetry group} (or isotropy group) is defined as
\begin{equation*}
  G_{\vv} := \set{g\in G;\; \rho(g)\vv=\vv}.
\end{equation*}

The concept of symmetry group allows to define an equivalence relation on $\VV$, which is \emph{coarser} than the relation ``to be in the same orbit'' and defined as follows: two vectors $\vv_{1}$ and $\vv_{2}$ \emph{have the same isotropy class} (or same symmetry class in mechanics~\cite{FV1996,FV1997}) if they have conjugate symmetry groups. In the following, we shall denote by
\begin{equation*}
  [H] := \set{g H g^{-1};\; g\in G}
\end{equation*}
the conjugacy class of the subgroup $H$ of $G$. To each conjugacy class $[H]$, where $H$ is a closed subgroup of $G$, corresponds the subset of $\VV$ defined by
\begin{equation*}
  \strata{H} := \set{\vv \in \VV;\; [G_{\vv}] = [H]}.
\end{equation*}
If this subset is not empty, $[H]$ is called an \emph{isotropy class} and $\strata{H}$ is the \emph{isotropy stratum associated to $[H]$}. It is known (see~\cite{Mos1957,Bre1960,Man1962}) that there is only a finite number of isotropy classes for any finite dimensional representation of a compact group.

The set of conjugacy classes $[H]$ of closed subgroups of a compact group is endowed with a partial order relation (reflexivity and transitivity are direct and true even if $G$ is not compact but anti-symmetry requires the compacity of $G$~\cite[Proposition 1.9]{Bre1972}), given by
\begin{equation*}
  [H]\preceq [K]\iff \exists g\in G,\quad gHg^{-1}\subset K.
\end{equation*}

Due to the order relation defined on the conjugacy classes, we define a \emph{closed stratum} to be the set consisting of vectors having \emph{at least the symmetry $[H]$, denoted by $\overline{\Sigma}_{[H]}$, and defined by}
\begin{equation*}
  \cstrata{H} = \set{\vv \in \VV;\; [H] \preceq [G_{\vv}]} = \bigcup_{[H] \preceq [K]} \Sigma_{[K]}.
\end{equation*}

The isotropy stratum $\Sigma_{[H]}$ and the closed isotropy stratum $\overline{\Sigma}_{[H]}$ are semialgebraic sets~\cite{AS1981,AS1983,PS1985,Sch1989}, \textit{i.e} defined by polynomial equations and inequalities~\cite{Cos2002,BCR2013}. Actually, if $G$ is a subgroup of $\GL(\VV)$, we can give a direct proof of this fact. Indeed, if $G$ is a compact subgroup of $\GL(\VV)$, $G$ is a real algebraic set by~\cite[Chapter 3, paragraph 4, Theorem 5]{OV1990}. Notice that so is the subset $G_{\vv}$, if $\vv \in \VV$, as it is described by polynomial equations in the coefficients of the matrices of the real algebraic set $G$. Now, since $H$ is a closed subgroup of $G$, $H$ is in particular a compact subgroup of $\GL(\VV)$ and then a real algebraic set as well. As a consequence, the closed isotropy stratum
\begin{equation*}
  \overline{\Sigma}_{[H]} = \set{\vv \in \VV;\; \exists g\in G, \forall h \in H, ghg^{-1}\vv = \vv},
\end{equation*}
and the isotropy stratum
\begin{equation*}
  \Sigma_{[H]} = \set{\vv \in \VV;\; \exists g\in G, \forall h \in H, ghg^{-1} \vv = \vv  \mbox{~~and~~} \exists g' \in G, \forall k \in G, k\vv = \vv \Rightarrow g'k{g'}^{-1} \in H},
\end{equation*}
are both described by \emph{first-order formulae} (in the sense of~\cite[Definition 2.2.3]{BCR2013}) so that the sets $\overline{\Sigma}_{[H]}$ and $\Sigma_{[H]}$ are semialgebraic sets by~\cite[Proposition 2.2.4]{BCR2013} (the latter cited result is an avatar of \emph{Tarski-Seidenberg theorem} which is an angular stone of semialgebraic geometry).

We shall introduce the distance of a vector $\vv_{0} \in \VV$ to the closed isotropy stratum $\cstrata{H}$
\begin{equation}\label{eq:distance}
  \Delta(\vv_{0},\cstrata{H})^{2}
  :=\underset{\vv\in\cstrata{H}}{\min}  \norm{\vv_{0}-\vv}^{2},
\end{equation}
for some $G$-invariant norm $\norm{\cdot}$. A minimizer will be denoted by $\vv^{*}$.

Examples of interest for the present work are provided by Continuum Mechanics, for which $G$ is either $\SO(3)$ or $\OO(3)$, $\VV$ is a space of tensors on $\RR^{3}$, endowed with the invariant norm
\begin{equation*}
  \norm{\bT}:=\sqrt{T_{i_{1} \dotsc i_{n}} T_{i_{1} \dotsc i_{n}}},
\end{equation*}
and the action on a tensor $\bT$ is written (in an orthonormal basis)
\begin{equation*}
  (\rho(g)\bT)_{i_{1} \dotsc i_{n}} := {g_{i_{1}}}^{j_{1} } \dotsc {g_{i_{n}}}^{j_{n}} T_{j_{1} \dotsc j_{n}},
  \quad
  \bT\in \VV, \;
  g\in G.
\end{equation*}
Finally, the $\OO(3)$-subgroups will be denoted according to the notations in~\cite{GSS1988}.

\begin{exmp}\label{ex:Ela}
  In elasticity,
  \begin{equation*}
    \VV=\Ela=\Sym^{2}(\Sym^{2}(\RR^{3}))=\set{\bE\in \otimes^{4}\RR^{3},\, E_{ijkl}=E_{jikl}=E_{klij}}
  \end{equation*}
  is the 21-dimensional vector space of elasticity tensors, and $G=\SO(3)$. In that case there are
  exactly eight isotropy classes $[\triv]$, $[\ZZ_{2}]$, $[\DD_{2}]$, $[\DD_{3}]$, $[\DD_{4}]$, $[\OO(2)]$, $[\octa]$, $[\SO(3)]$~\cite{FV1996},
  and the problem of the distance to an elasticity isotropy stratum has been investigated in~\cite{Via1997,FGB1998,Del2005,MN2006,KS2008,KS2009}.
\end{exmp}

\begin{exmp}\label{ex:Piez}
  In piezoelectricity,
  \begin{equation*}
    \VV=\Piez=\set{\be\in \otimes^{3}\RR^{3},\, \re_{ijk}=\re_{ikj}}
  \end{equation*}
  is the 18-dimensional vector space of piezoelectricity tensors,	and $G=\OO(3)$. In that case there are exactly 16 isotropy classes
  $[\triv]$, $[\ZZ_{2}]$, $[\ZZ_{3}]$, $[\DD_{2}^{z}]$, $[\DD_{3}^{z}]$, $[\ZZ_{2}^{-}]$, $[\ZZ_{4}^{-}]$,   $[\DD_{2}]$, $[\DD_{3}]$, $[\DD_{4}^{d}]$, $[\DD_{6}^{d}]$, $[\SO(2)]$, $[\OO(2)]$, $[\OO(2)^{-}]$, $[\octa^{-}]$, $[\OO(3)]$~\cite{Nye1985,ZB1994,OA2021},
  and the problem of the distance to a piezoelectricity isotropy stratum has been investigated in~\cite{ZTP2013}.
\end{exmp}

When polynomial equations and/or inequalities characterizing the semialgebraic set $\cstrata{H}$ are known (see~\cite{AKP2014,OKDD2021}), the distance to an isotropy stratum problem~\eqref{eq:distance} reduces to minimize a polynomial function (the quadratic function $\Delta(\, \cdot \,,\cstrata{H})^{2}$) under polynomial constraints. In that case, we can solve the distance to an isotropy stratum problem using \emph{polynomial and semialgebraic optimization}~\cite{Las2015,Lau2009,Mev2007,Sch2005}, which allows to approximate numerically the global minimum of the function $\Delta(\, \cdot \,,\cstrata{H})^{2}$.

\section{Semialgebraic optimization method}
\label{sec:poly-opt}

The problem of determining the constitutive tensor having a specific symmetry the closest to an experimental one can be viewed as an example of the problem of minimizing a polynomial function over polynomial constraints
\begin{equation}\label{eq:pmin}
  f^{*} = \inf\set{f(\xx);\; \xx\in K},
\end{equation}
where $f\in\RR[X]:=\RR[X_{1},\dotsc,X_n]$, $\xx = (x_{1},\ldots,x_n) \in \RR^{n}$ and $K$ is a basic closed semialgebraic set
\begin{equation*}
  K = \set{\xx \in \RR^{n};\; g_{1}(\xx) \ge 0, \dotsc ,g_{m}(\xx)\ge 0},
\end{equation*}
with $g_{1},\dotsc, g_{m} \in \RR[X]$ (see~\cite{Cos2002,BCR2013} for self-contained references on semialgebraic geometry).

We now describe \emph{Lasserre's method}~\cite{Las2001,Las2015}, that will allow us to solve numerically the problem of the distance from an experimental tensor to a closed stratum. The method consists in constructing a sequence of semidefinite programs whose optimal values form a nondecreasing sequence which converges to the optimum $f^{*}$.

In this section, apart from theorem~\ref{thm:rank_condition} which is a refinement of~\cite[Theorem 6.2]{Las2015}, there is no original statement: we give the essential steps and results of the approach for pedagogical reasons and to be self-contained. For more details on Lasserre's method and the involved mathematical results and background, we refer to~\cite{Las2015,Lau2009,Sch2005}.

The first step of the method is to notice that the optimization problem \eqref{eq:pmin} can be reformulated as
\begin{equation}\label{eq:pprob}
  f^{*}=\inf\set{\int_{\RR^{n}} f {\rm d}\mu ;\; \mu\ \text{probability measure on $\RR^{n}$ with support in $K$}}.
\end{equation}
Indeed, for $\xx\in K$, if $\delta_{\xx}$ denotes the probability Dirac measure on $\RR^{n}$ at $\xx$, we have $f(\xx)=\int f {\rm d}\delta_{\xx}$ and, conversely, if $\mu$ is a probability measure on $\RR^{n}$ with support in $K$,
\begin{equation*}
  \int_{\RR^{n}} f(\xx){\rm d}\mu(\xx) \ge \int_{\RR^{n}} f^{*} {\rm d}\mu(\xx)=f^{*}.
\end{equation*}

Now, if $(y_{\alpha})_{\alpha\in \mathbb{N}^{n}}$ is a sequence of real numbers, denote by $M(y)$ the infinite symmetric matrix
\begin{equation*}
  \left(y_{\alpha+\beta}\right)_{\alpha, \beta \in \NN^{n}},
\end{equation*}
(called the \emph{moment matrix associated to $y$}) and, if $g = \sum_{\beta \in \NN^{n}} g_{\beta} X^{\beta}$, set
\begin{equation*}
  g \cdot y := \left(\sum_{\beta \in \NN^{n}} g_{\beta} y_{\alpha + \beta}\right)_{\alpha \in \NN^{n}} \in \RR^{\NN^{n}}.
\end{equation*}
Under an hypothesis called the \emph{Archimedean hypothesis}, we can write the optimization problem (\ref{eq:pprob}) as
\begin{equation}\label{eq:pmom}
  f^{*} = \inf_{y \in \RR^{\NN^{n}}}\set{ \langle f,y\rangle ;\; y_{(0,\ldots,0)} = 1, \, M\left(g_i \cdot y \right) \succeq 0, \, i = 0,\ldots,m}
\end{equation}
where, if $y \in \RR^{\NN^{n}}$, $\langle f,y\rangle := \sum_{\alpha} f_{\alpha} y_{\alpha}$, and $g_{0} := 1$. Here, if $M$ is a finite or infinite matrix with real coefficients, $M \succeq 0$ means that $M$ is positive semidefinite (an infinite symmetric matrix is called positive semidefinite if all its principal submatrices are positive semidefinite). The formulation~\eqref{eq:pmom} is a direct consequence of the following solution of the moment problem on~$K$.

\begin{thm}[Putinar, Jacobi--Prestel]\label{prop:sol-mom-pb-archi}
  Suppose that the polynomials $g_{1},\ldots,g_{m}$ describing $K$ satisfy the Archimedean hypothesis. Then, for all $y \in \RR^{\NN^{n}}$, $y$ has a representative measure on $K$ (i.e. there exists a finite Borel measure $\mu$ on $\RR^{n}$ with support in $K$ such that, for any $\alpha \in \mathbb{N}^{n}$, $y_{\alpha}=\int_{\RR^{n}} \xx^{\alpha} {\rm d}\mu(\xx)$) if and only if the moment matrices $M(y)$, $M(g_{1} \cdot y), \ldots, M(g_{m} \cdot y)$ are positive semidefinite.
\end{thm}

This statement is a reformulation of~\cite[theorem 2.44]{Las2015} and is due to Putinar (\cite{Put1993}) and Jacobi--Prestel (\cite{JP2001}). We recall the proof below but, first, we have to define the essential \emph{Archimedean hypothesis}.

\begin{defn}\label{def:archimedean_prop}
  Consider the $\RR$-module
  \begin{equation*}
    \mathbf{M}(g_{1},\dotsc,g_{m}):=\set{\sigma_{0} + \sum_{i = 1}^{m} \sigma_{i} g_{i};\; \sigma_{0}, \dotsc, \sigma_{m} \text{ sum of squares}} \subset \RR[X]
  \end{equation*}
  (a polynomial $p \in \RR[X]$ is a \emph{sum of squares} if there exist polynomials $p_{1},\ldots,p_N \in \RR[X]$ such that $p = \sum_{j=1}^N p_j^{2}$). We say that the polynomials $g_{1},\ldots,g_{m}$ satisfy the \emph{Archimedean hypothesis} (or that $\mathbf{M}(g_{1},\dotsc,g_{m})$ is an \emph{Archimedean module}) if there exists a positive integer $N$ such that
  \begin{equation*}
    N - \sum_{i=1}^{n} X_{i}^{2} \in  \mathbf{M}(g_{1},\dotsc,g_{m}).
  \end{equation*}
\end{defn}

We refer to~\cite[Theorem 1.1]{Sch2005} (a result due to Schmüdgen) for a list of properties equivalent to the Archimedean hypothesis. Notice that if $g_{1},\ldots,g_{m}$ satisfy the Archimedean hypothesis, then
\begin{equation*}
  K = \set{\xx \in \RR^{n};\; g_{1}(\xx) \ge 0, \dotsc ,g_{m}(\xx)\ge 0}
\end{equation*}
is necessarily compact. The crucial point is that, if $g_{1},\ldots,g_{m}$ satisfy the Archimedean hypothesis, then we have access to Putinar's Positivstellensatz.

\begin{thm}[Putinar~\cite{Put1993}]\label{thm:putinar_pos}
  Suppose that the polynomials $g_{1},\ldots,g_{m}$ describing $K$ satisfy the Archimedean hypothesis, and let $p \in \RR[X]$. If $p(K) \subset ]0;+\infty[$, then $p \in \mathbf{M}(g_{1},\ldots,g_{m})$.
\end{thm}

See also~\cite[section 2]{Sch2005} and~\cite[section 3.7]{Lau2009} for alternative proofs.

\begin{proof}[Proof of theorem~\ref{prop:sol-mom-pb-archi}]
  Let $y \in \RR^{\NN^{n}}$. The direct implication is actually true even if the polynomials $g_{1},\ldots,g_{m}$ do not satisfy the Archimedean hypothesis. Indeed, suppose that $y$ has a representing measure $\mu$ on $K$. Now, take any vector $\bp = \left(p_{\alpha}\right)_{\alpha  \in \NN^{n}}$ of $\RR^{\NN^{n}}$ with finitely many nonzero coordinates and set $p := \sum_{\alpha} p_{\alpha} X^{\alpha}$. If $M$ is any matrix, denote by $M^T$ its transpose. If $g$ is any polynomial of $\set{g_{0},\ldots,g_{m}}$, we have then
  \begin{equation*}
    \bp^{T}M\left(g \cdot y\right)\bp =  \sum_{\alpha,\beta} p_{\alpha}p_{\beta} \left(\sum_{\gamma} g_{\gamma} \int_{\RR^{n}} \xx^{\alpha+\beta+\gamma}{\rm d}\mu(\xx)\right) = \int_{\RR^{n}} g(\xx) p(\xx)^{2} {\rm d} \mu(\xx) \geq 0.
  \end{equation*}
  By definition, for all $\xx \in K$, $g(\xx) \geq 0$, and the support of $\mu$ is included in $K$.

  Conversely, suppose that the matrices $M(y)$, $M(g_{1} \cdot y), \ldots, M(g_{m} \cdot y)$ of $y$ are positive semidefinite, and denote by $L_y$ the linear mapping
  \begin{equation*}
    p = \sum_{\alpha} p_{\alpha} X^{\alpha} \in \RR[X] \mapsto  \sum_{\alpha} p_{\alpha} y_{\alpha} \in \RR.
  \end{equation*}
  Let $g \in \set{g_{0},\ldots,g_{m}}$ and consider the symmetric bilinear form
  \begin{equation*}
    \begin{array}{ccl}
      \RR[X] \times \RR[X] & \rightarrow & \RR                 \\
      (p,q)                & \mapsto     & L_y\left(pqg\right)
    \end{array},
  \end{equation*}
  which is represented by the localizing matrix $M\left(g \cdot y\right)$ in the canonical basis of $\RR[X]$. In particular, for every polynomial $p$, if $\bp$ denotes the vector $\left(p_{\alpha}\right)_{\alpha \in \NN^{n}}$, we have
  \begin{equation*}
    L_y\left(p^{2} g\right) = \bp^{T}M\left(g \cdot y\right)\bp \geq 0
  \end{equation*}
  and, consequently, the linear mapping $L_y$ has nonnegative values on $\mathbf{M}(g_{1},\ldots,g_{m})$. By Putinar's Positivstellensatz~\ref{thm:putinar_pos}, this implies that $L_y$ has nonnegative values on any polynomial $p \in \RR[X]$ such that $p(K) \subset ]0;+\infty[$.

  If $p \in \RR[X]$ satisfies $p(K) \subset [0;+\infty[$ then, for any positive real number $\epsilon$, the polynomial $p+\epsilon$ has positive values on $K$ so that $L_y(p)+ \epsilon = L_y(p+\epsilon) \geq 0$, and therefore $L_y(p) \geq 0$. We can then apply Haviland's theorem (\cite{Hav1935}, see also~\cite[section 3.2]{Mar2008} and Theorem 4.15 and section 4.6 of the up-to-date version of~\cite{Lau2009}) to the mapping $L_y$ : there exists a measure $\mu$ on $\RR^{n}$ with support in $K$ such that $L_y(p)=\int_{\RR^{n}} p(\xx){\rm d}\mu(\xx)$ for all $p\in \RR[X]$. In particular, for all $\alpha \in \NN^{n}$, we have
  \begin{equation*}
    y_{\alpha} = L_y\left(X^{\alpha}\right) = \int_{\RR^{n}} \xx^{\alpha}{\rm d}\mu(\xx)
  \end{equation*}
  \textit{i.e.}, $y$ is the moment sequence of the measure $\mu$.
\end{proof}

From now on, we assume that the polynomials $g_{1},\dotsc,g_{m}$ satisfy the Archimedean hypothesis so that we can write
\begin{equation}\label{eq:pmom2}
  f^{*} = \inf_{y \in \RR^{\NN^{n}}}\set{ \langle f,y\rangle ;\; y_{0} = 1, \, M\left(g_i \cdot y \right) \succeq 0, \, i=0,\dotsc,m},
\end{equation}
where $y_{0} := y_{(0,\ldots,0)}$.

Lasserre's method to solve the optimization problem~\eqref{eq:pmom} consists, then, in \emph{relaxing} this infinite-dimensional problem into a sequence of finite-dimensional problems which are \emph{semidefinite programs}. \emph{Semidefinite programs}, or SDP's, are optimization problems over finite positive semidefinite symmetric matrices which generalize linear programs, and for which there exist efficient algorithms of numerical resolution. SDP-solving algorithms include methods inspired by the ones used in linear programming, such as interior point methods (see for instance the references~\cite{Ali1995,VB1996,Stu1997,Tod2001,Fre2004,WSV2012}).

Below, we follow Lasserre's notations in~\cite[section 6.1.1]{Las2015}. First, if $k \in \NN$, let
\begin{equation*}
  \Lambda(k):=\set{\left(\alpha_{1},\ldots,\alpha_n\right)\in \NN^{n}; \alpha_{1}+\cdots+ \alpha_n \leq k}
\end{equation*}
and, if $y \in \RR^{\Lambda(2k)}$ and $k' \in \NN$ satisfies $k' \leq k$, set
$\displaystyle{M_{k'}(y) := \left(y_{\alpha+\beta}\right)_{\alpha, \beta \in \Lambda(k')}}$. If $g \in \RR[X]$, set
\begin{equation*}
  g \cdot y := \left(\sum_{\beta \in \NN^{n}} g_{\beta} y_{\alpha + \beta}\right)_{\alpha \in \Lambda(k)} \in \RR^{\Lambda(k)}.
\end{equation*}
Finally, for $i \in \set{0,\ldots,n}$, denote $v_i := \ceil*{\frac{\deg(g_i)}{2}}$ (notice that $v_{0} = 0$) and let $d_{0}$ be the integer $\max\left(\ceil*{\frac{\deg(f)}{2}},v_{1},\ldots,v_{m}\right)$.

For $d$ any integer such that $d \geq d_{0}$, we then consider the optimization problem
\begin{equation}\label{eq:prelaxed}
  \rho_{d} = \inf_{y \in \RR^{\Lambda(2d)}}\set{ \langle f,y\rangle ;\; y_{0} = 1, \, M_{d-v_i}\left(g_i \cdot y \right) \succeq 0, \, i = 0,\ldots,m},
\end{equation}
relaxed from (\ref{eq:pmom2}).

For a given $d \geq d_{0}$, the optimization problem (\ref{eq:prelaxed}) is a semidefinite program (and can then be numerically solved using SDP solvers). Indeed, for all $y \in \RR^{\Lambda(2d)}$ such that $y_{0} = 1$ and all $i \in \set{0,\ldots,m}$, we can write
\begin{equation*}
  M_{d-v_i}\left(g_i \cdot y \right) = A_{0 \, i} + \sum_{\alpha \in \Lambda(2d)\setminus \set{0}} y_{\alpha} A_{\alpha \, i}
\end{equation*}
where, for all $\alpha \in \Lambda(2d)$, $A_{\alpha \, i}$ is a symmetric square matrix of size $\Lambda(d-v_i)$ (see also~\cite[section~5]{Sch2005}).

The following theorem of Lasserre (\cite[Theorem 6.2]{Las2015}, see also~\cite[Theorem 1.5]{Sch2005}) asserts that the sequence of optima $\left(\rho_{d}\right)_{d \geq d_{0}}$ converges to $f^{*}$:

\begin{thm}[Lasserre]\label{thm:Lasserre}
  The sequence $\left(\rho_{d}\right)_{d\geq d_{0}}$ is a nondecreasing sequence that converges to~$f^{*}$.
\end{thm}

\begin{proof}
  Let $d$ be an integer such that $d \geq d_{0}$ and denote by $F_{d}$ the set of vectors $y \in \RR^{\Lambda(2d)}$ such that $y_{0} = 1$ and $M_{d-v_i}\left(g_i \cdot y \right) \succeq 0$ for all $i \in \set{0,\ldots,m}$. The set $\set{\langle f,y\rangle;\; y \in F_{d+1}}$ is included in the set $\set{\langle f,y\rangle;\; y \in F_{d}}$. Indeed, if $y \in F_{d+1}$ and if we denote by $\overline{y}$ the truncation $\left(y_{\alpha}\right)_{\alpha \in \Lambda(2d)}$ of $y$, we have $\overline{y}_{0} = y_{0} = 1$ and, for $i \in \set{0,\ldots,m}$, $M_{d-v_i}\left(g_i \cdot \overline{y} \right) \succeq 0$ (because $M_{d-v_i}\left(g_i \cdot \overline{y} \right)$ is a principal submatrix of the positive semidefinite matrix $M_{d+1-v_i}\left(g_i \cdot y \right)$), as well as $\langle f, \overline{y} \rangle = \langle f, y\rangle$ since $\deg f \leq 2 d$. As a consequence, $\rho_{d} \leq \rho_{d+1}$.

  We then show that the nondecreasing sequence $\left(\rho_{d}\right)_{d\geq d_{0}}$ is bounded by $f^{*}$. Consider the formulation (\ref{eq:pmom2}) of our optimization problem and denote by $F$ the set of sequences $y \in \RR^{\NN^{n}}$ such that $y_{0} = 1$ and $M\left(g_i \cdot y \right) \succeq 0$ for all $i \in \set{0,\ldots,m}$. Let $y$ be in $F$ and let $\overline{y} := \left(y_{\alpha}\right)_{\alpha \in \Lambda(2d)}$ be the truncation of $y$. Again, we have $\overline{y}_{0} = 1$, $M_{d-v_i}\left(g_i \cdot \overline{y} \right) \succeq 0$, $i \in \set{0,\ldots,m}$, and $\langle f, \overline{y} \rangle = \langle f, y\rangle$, so that $\rho_{d} \leq f^{*}$. Therefore, the sequence $\left(\rho_{d}\right)_{d\geq d_{0}}$ converges.

  The last step is to show that $f^{*}$ is actually the limit of $\left(\rho_{d}\right)_{d\geq d_{0}}$. If $\epsilon$ is a positive real number, one can show that there exists $d \geq d_{0}$ such that $f^{*}-\epsilon \leq \rho_{d} \leq f^{*}$: the interested reader is invited to refer to~\cite[Theorem 6.2]{Las2015} or~\cite[Theorem 1.5]{Sch2005}. The proof involves the dual SDP associated to (\ref{eq:prelaxed}), together with Putinar's Positivstellensatz~\ref{thm:putinar_pos}.
\end{proof}

\section{Lasserre's algorithm -- GloptiPoly}
\label{sec:algo}

The principle of Lasserre's algorithm to solve problem (\ref{eq:pmin}) is to numerically compute the sequence of optima $\left(\rho_{d}\right)_{d \geq d_{0}}$ (which by theorem~\ref{thm:Lasserre} converges to $f^{*}$) using SDP solvers at each step. In order to complete this approach, one has to define a stopping criterion for the algorithm. In~\cite[section 6.1]{Las2015}, Lasserre chooses a sufficient condition in terms of ranks of moment matrices, a condition which is motivated by the theorem below. The result we show is actually a slight generalization of~\cite[Theorem 6.6]{Las2015}, that we decided to state in order to take into account the fact that a SDP solver, when applied to the SDP (\ref{eq:prelaxed}), only provides, at best, a numerical approximation of the optimum $\rho_{d}$. Let $\epsilon$ be a nonnegative real number, $d$ be an integer such that $d \ge d_{0}$ and denote $v := \max(v_{1},\ldots,v_{m})$.

\begin{thm}\label{thm:rank_condition}
  Let $y \in F_{d}$ (we defined $F_{d}$ in the proof of theorem~\ref{thm:Lasserre}) such that $\rho_{d}\leq \langle f, y \rangle \leq \rho_{d}+\epsilon$. If $\mathrm{rank} \, M_{d-v}(y)= \mathrm{rank} \, M_{d}(y)$ then
  \begin{equation*}
    f^{*}\leq \langle f, y \rangle \leq\rho_{d} + \epsilon \leq f^{*} +\epsilon.
  \end{equation*}
  Moreover, if we denote $s :=  \mathrm{rank} \, M_{d}(y)$, there exist at least $s$ points $\xx$ of $K$ such that $f^{*} \leq f(\xx) \leq f^{*} + \epsilon$.
\end{thm}

In other words, if an optimal solution $y$, up to a fixed precision $\epsilon$, of the SDP (\ref{eq:prelaxed}) satisfies the above rank condition on its moment matrix, then $\langle f,y \rangle$ is an approximation of $f^{*}$ up to precision~$\epsilon$. Furthermore, there exist at least $\mathrm{rank} \, M_{d}(y)$ points of $K$ which are global minimizers of $f$ up to precision $\epsilon$.

\begin{rem}
  \begin{enumerate}
    \item For $\epsilon = 0$, we recover Theorem 6.6 in~\cite{Las2015}.

    \item The SDP solver used in the algorithm of Lasserre implemented in the freeware \emph{GloptiPoly~3} computes an element $y$ of $F_{d}$ which is an approximation of an optimal solution of (\ref{eq:prelaxed}) and such that the rank $r$ of $M_{d}(y)$ is maximal among the ranks of moment matrices of elements of $F_{d}$. GloptiPoly then checks if the \emph{numerical rank} of the principal submatrix $M_{d-v}(y)$ of $M_{d}(y)$ is equal to $r$. The \emph{numerical rank} of a matrix $M$ is, roughly speaking, the number of singular values of $M$ which are greater than a fixed precision, and the numerical rank of $M$ is not greater than its rank. As a consequence, if the numerical rank of $M_{d-v}(y)$ is (at least) $r$, we have the inequalities $r \leq \mathrm{rank} \, M_{d-v}(y) \leq \mathrm{rank} \, M_{d}(y) =r$ so that $\mathrm{rank} \, M_{d-v}(y) = \mathrm{rank} \, M_{d}(y)$ and the stopping criterion of theorem~\ref{thm:rank_condition} applies. More details about these questions can be found in~\cite[sections 4.4.1 and 4.4.2]{JLR2008}.

    \item In~\cite[section 6.1.2]{Las2015} is described the algorithm, implemented in GloptiPoly, which extract (approximated) global minimizers of $f$ when the rank condition is satisfied.
  \end{enumerate}
\end{rem}

Theorem~\ref{thm:rank_condition} is a consequence of the following one whose sketch of proof is postponed below. For any $r \in \NN \setminus \set{0}$, a Borel measure $\mu$ on $\RR^{n}$ is said to be {\it $r$-atomic} if there exist $\xx_{1},\ldots,\xx_r \in \RR^{n}$ and positive real numbers $\lambda_{1},\ldots,\lambda_r$ such that $\mu = \sum_{i=1}^r \lambda_i \delta_{\xx_i}$.

\begin{thm}[Curto--Fialkow~\cite{CF2000}, Laurent~\cite{Lau2005}]\label{thm:atom-meas}
  Let $y \in F_{d}$. If $\mathrm{rank} \, M_{d-v}(y)= \mathrm{rank} \, M_{d}(y)$, then $y$ can be represented by a $s$-atomic measure, where $s :=  \mathrm{rank} \, M_{d}(y)$, whose support is included in $K$.
\end{thm}

\begin{proof}[Proof of theorem~\ref{thm:rank_condition}]
  We adapt the proof of~\cite[theorem 6.6]{Las2015}. Suppose that rank $M_{d-v}(y)=\text{rank }M_{d}(y)$. Then, by theorem~\ref{thm:atom-meas}, $y$ has a $s$-atomic representing measure $\mu$ with support included in $K$ : there exist $\xx_{1},\ldots,\xx_s \in K$ and $\lambda_{1},\ldots,\lambda_s \in ]0;+\infty[$ such that $\mu = \sum_{i=1}^s \lambda_i \delta_{\xx_i}$. In particular, since $y_{0} = 1$, we have $1 = y_{0} = \int_{\RR^{n}} \mu(\xx) = \sum_{i=1}^s \lambda_i$. Then
  \begin{equation*}
    \langle f, y \rangle = \sum_{\alpha \in \Lambda(2d)} f_{\alpha} y_{\alpha} = \int_{\RR^{n}} \sum_{\alpha \in \Lambda(2d)} f_{\alpha} \xx^{\alpha} {\rm d}\mu(\xx) = \sum_{i=1}^s \lambda_i f(\xx_i) \geq \sum_{i=1}^s \lambda_i f^{*} = f^{*},
  \end{equation*}
  so that $f^{*}+\epsilon \ge \rho_{d}+\epsilon \ge \langle f, y \rangle \ge f^{*}$.

  Finally suppose that there exists $i \in \set{1,\dots,s}$ such that $f(\xx_i) > f^{*} + \epsilon$. This implies that
  \begin{equation*}
    \sum_{j=1}^s \lambda_j f(\xx_j) > \left(\sum_{j =1}^s \lambda_j f^{*} \right) + \epsilon = f^{*} + \epsilon,
  \end{equation*}
  which is not true according to the above inequalities. As a consequence, for all $i \in \set{1,\dots,s}$, $f^{*} \leq f(\xx_i) \leq f^{*} + \epsilon$.
\end{proof}

\begin{proof}[Proof of theorem~\ref{thm:atom-meas}]
  We point out the essential steps of the reasoning, referring to~\cite{Lau2009} for the detailed proofs. We have
  \begin{equation*}
    \mathrm{rank} \, M_{d-v}(y) \leq \mathrm{rank} \, M_{d-v +1}(y) \leq \cdots \leq \mathrm{rank} \, M_{d-1}(y) \leq \mathrm{rank} \, M_{d}(y)
  \end{equation*}
  and suppose that $\mathrm{rank} \, M_{d-v}(y)= \mathrm{rank} \, M_{d}(y)$: we obtain that $\mathrm{rank} \, M_{d-1}(y)= \mathrm{rank} \, M_{d}(y)$. We can then recursively apply the Flat Extension Theorem 5.14 of~\cite{Lau2009} (originally due to Curto and Fialkow in~\cite{CF1996}) to assert the existence of a sequence $\widetilde{y}$ of $\RR^{\NN^{n}}$ such that, for all $\alpha \in \Lambda(2d)$, $\widetilde{y}_{\alpha} = y_{\alpha}$ and, for all $k \in \NN$ satisfying $k \geq d-v$, $\mathrm{rank} \, M_k\left(\widetilde{y}\right) = \mathrm{rank} \, M_{d}(y) = s$.

  In particular, for all $k \geq d$, since the principal submatrix $M_{d}(y) = M_{d}\left(\widetilde{y}\right)$ of $M_k\left(\widetilde{y}\right)$ is positive semidefinite (because $y$ is in $F_{d}$) and $\mathrm{rank} \, M_k\left(\widetilde{y}\right) = \mathrm{rank} \, M_{d}\left(\widetilde{y}\right)$, the symmetric matrix $M_k\left(\widetilde{y}\right)$ is also positive semidefinite (see~\cite[Definition 1.1]{Lau2009}). In other words, the (infinite) moment matrix $M\left(\widetilde{y}\right)$ is positive semidefinite. Since, furthermore, $\mathrm{rank} \, M\left(\widetilde{y}\right) = s$, by~\cite[Theorem 5.1 (i)]{Lau2009}, there is a $s$-atomic measure $\mu$ representing $\widetilde{y}$, and then $y$, with support the finite real algebraic set
  \begin{equation*}
    V(I) := \set{\xx \in \RR^{n} ;\; \mbox{for all $p \in I$}, p(\xx) = 0}
  \end{equation*}
  where $I := \set{p \in \RR[X];\; M p = 0}$ (the proof of Theorem 5.1 (i) of~\cite{Lau2009} involves real algebraic geometry).

  The last step is then to prove that this support is included in $K$. Write $V(I) = \set{\xx_{1},\ldots,\xx_s}$ and let $\lambda_{1},\ldots, \lambda_s \in ]0;+\infty[$ such that $\mu = \sum_{i=1}^s \lambda_i \delta_{\xx_i}$. By Lemma 5.6 of~\cite{Lau2009}, there exist $p_{1},\ldots,p_s \in \RR[X]$ of degree at most $d-v$ such that, for all $i,j \in \set{1,\ldots,s}$, $p_i(\xx_j) = \delta_{ij}$ (see also~\cite[Lemma 2.3]{Lau2009}). For all $i \in \set{1,\ldots,s}$ and $j \in \set{1,\ldots,m}$, we then have, because $y \in F_{d}$,
  \begin{equation*}
    0 \leq p_i^T M_{d-v}(g_j \cdot y) p_i = \int_{\RR^{n}} g_j(\xx) p_i(\xx)^{2} {\rm d}\mu(\xx) = \sum_{k=1}^s \lambda_k g_j(\xx_k) p_i(\xx_k)^{2} = \lambda_i g_j(\xx_i)
  \end{equation*}
  (see the proof of proposition~\ref{prop:sol-mom-pb-archi} above for the first equality) and, since $\lambda_i > 0$, $g_j(\xx_i) \geq 0$. As a consequence, for all $i \in \set{1,\ldots,s}$, $\xx_i \in \set{\xx \in \RR^{n};\; g_{1}(\xx) \ge 0, \dotsc ,g_{m}(\xx)\ge 0}=K$.
\end{proof}

We finally present the algorithm implemented by Lasserre and Henrion in the Matlab freeware \emph{GloptiPoly 3} to numerically solve polynomial optimization problems. For details on GloptiPoly and its use, see~\cite{HLL2009} and Appendix B of~\cite{Las2015}.

In order to solve a SDP relaxation $\rho_{d}$, $d \geq d_{0}$, GloptiPoly 3 uses by default the SDP solver SeDuMi of Sturm~\cite{Stu1999}. Other SDP solvers can also be used as long as they are interfaced through Yalmip~\cite{Lof2004} (see section 5.9 of ~\cite{HLL2009}).

The inputs of GloptiPoly are
\begin{itemize}
  \item the variables $X_{1},\ldots,X_n$,
  \item the polynomial $f \in \RR[X_{1},\ldots,X_n]$,
  \item the polynomials $g_j \in \RR[X_{1},\ldots,X_n]$, $j \in \set{1,\ldots,m}$, satisfying the Archimedean hypothesis,
  \item a maximal relaxation order $d_{\text{max}} \geq d_{0}$.
\end{itemize}

The first output is a \emph{status} number $\xi \in \set{-1,0,1}$:
\begin{itemize}
  \item $\xi = -1$ means that the consider SDP solver could not solve numerically any of the relaxations $\rho_{d}$, $d \in \set{d_{0},\ldots,d_{\text{max}}}$;
  \item $\xi = 0$ means that the solver numerically solved (that is up to a prescribed precision $\epsilon$) at least one of the relaxations $\rho_{d}$, $d \in \set{d_{0},\ldots,d_{\text{max}}}$, but at each such success either no optimal solution was provided by the solver, either the rank stopping criterion of Theorem~\ref{thm:rank_condition} was not satisfied by the (approximated up to precision $\epsilon$) obtained optimal solution $y_{d}$. In that case, the algorithm also outputs the last computed (and then greatest) optimal value $\rho_{d}$ which is (up to precision $\epsilon$) a lower bound for $f^{*}$;
  \item $\xi = 1$ means that the rank stopping criterion of Theorem~\ref{thm:rank_condition} has been satisfied by an optimal solution $y_{d}$ of a solved relaxation $\rho_{d}$, $d \in \set{d_{0},\ldots,d_{\text{max}}}$. In that case, the algorithm also outputs $\rho_{d}$ which is then an approximation of $f^{*}$ up to the prescribed precision $\epsilon$.
\end{itemize}

Lasserre's algorithm for polynomial optimization is, in pseudo code, the following (see~\cite[Algorithm 6.1]{Lau2009}):

\begin{algorithmic}
  \State $d \gets d_{0}$
  \State $\rho \gets -\infty$
  \State $\xi \gets -1$
  \While{$d \leq d_{\text{max}}$}
  ask SDP solver to solve $\rho_{d}$
  \If{not possible}
  \State $d \gets d+1$
  \Else
  \State $\xi \gets 0$
  \State $\rho \gets \text{optimal value provided by the solver}$
  \If{SDP solver found an optimal solution $y$ and $y$ satisfies rank stopping criterion}
  \State $\xi \gets +1$
  \State \textbf{Return} $\xi$,$\rho$ 
  \State \textbf{Stop}
  \Else
  \State $d \gets d+1$
  \EndIf
  \EndIf
  \EndWhile
  \State \textbf{Return} $\xi$,$\rho$
\end{algorithmic}

\begin{rem}
  \begin{enumerate}
    \item When the rank condition is satisfied, we can also ask GloptiPoly to extract minimizers up to precision $\epsilon$ (in the sense of Theorem~\ref{thm:rank_condition}), which involves the algorithm described in~\cite[section 6.1.2]{Las2015}.
    \item If the output $\xi$ is $0$ or $-1$, one can increase $d_{\text{max}}$ to try to obtain an approximation (or a better lower bound) of $f^{*}$ at a higher relaxation order.
    \item There is no complexity known for Lasserre's method. Actually, we do not know if there is a maximal relaxation degree $d_{\text{max}}$, dependent on the inputs of the problem, which would ensure the rank stopping criterion to be satisfied at some ordre $d \leq d_{\text{max}}$. However, what makes this method advantageous is that it benefits from the interesting complexity of SDP solvers to solve semidefinite programs (see for instance~\cite{Las2015} A.1.2).
  \end{enumerate}
\end{rem}

We conclude this part by the following remark: the convergence of Lasserre's polynomial optimization method, described in the previous sections, takes place when the constraint set is a semialgebraic compact set satisfying the Archimedean property~\ref{def:archimedean_prop}. Nevertheless, when the Archimedean condition is not satisfied but the polynomial function $f$ is coercive, Jeyakumar--Lasserre--Li in~\cite{JLL2014} provide a way to consider the optimization problem (\ref{eq:pmin}) as a problem with constraints satisfying the Archimedean condition :

\begin{lem}[Jeyakumar--Lasserre--Li]\label{lem:coercivity}
  Suppose that the polynomial function $f: \RR^{n} \rightarrow \RR$ associated to $f \in \RR[X_{1},\dotsc,X_n]$ is coercive, and let $c>0$ and $\yy \in K$ such that $c>f(\yy)$. Then the quadratic module $\mathbf{M}(g_{1},\dotsc,g_{m},c-f)$ associated to the semialgebraic set
  \begin{equation*}
    \widetilde{K}=\set{\xx\in \RR^{n};\; g_{1}(\xx)\ge 0,\dotsc,g_{m}(\xx)\ge 0,c-f(\xx)\ge 0}
  \end{equation*}
  is Archimedean (in particular, $\widetilde{K}$ is compact). Furthermore,
  \begin{equation*}
    f^{*} = \inf_{\xx \in K} f(\xx) = \inf_{\xx \in \widetilde{K}} f(\xx) =  \min_{\xx \in \widetilde{K}} f(\xx).
  \end{equation*}
\end{lem}

\begin{proof}
  The set $E=\set{\xx\in \RR^{n};\; c-f(\xx)\ge 0}$ is not empty since $\yy\in E$. Furthermore, the set~$E$ is compact. Indeed, if we suppose that $E$ is not bounded, we can find a sequence $(x_n)_{n\in \NN}$ of elements of $E$ such that $\norm{x_n} \to + \infty$. But then $f(x_n) \to + \infty$ since $f$ is coercive, which is impossible since, by definition of $E$, for all $n \in \NN,\  f(x_n) \le c$. Since $E=\set{\xx\in \RR^{n};\; c-f(\xx)\ge 0}$ is compact, then the quadratic module  $\mathbf{M}(g_{1},\dotsc,g_{m},c-f)$ is Archimedean by~\cite[Theorem 1.1]{Sch2005}. Finally, we have $c > f(\yy)$ so $c > f^{*}$ and
  \begin{equation*}
    f^{*} = \inf \set{f(\xx);\; \xx \in K} = \inf \set{ f(\xx) ;\; \xx \in K \cap \set{c - f \ge 0}\ }  = \inf \set{f(\xx);\; \xx \in \widetilde{K}}.
  \end{equation*}
\end{proof}

In other words, even if the polynomials $g_{1},\ldots,g_{m}$ do not satisfy themselves the Archimedean hypothesis, provided that $f$ is coercive, we can place ourselves in the range of application of Lasserre's method by adding the inequality $c - f \ge 0$ to the constraints $g_{1} \geq 0,\ldots,g_{m} \geq 0$.

\begin{rem}\label{rem:global_optimality}
  In~\cite[theorem 6.5]{Las2015}, Lasserre states some classical conditions (known as the Karush–Kuhn–Tucker (KKT) conditions~\cite[section 7.1]{Las2015}), already encountered in nonlinear programming, to ensure the finite convergence of the hierarchy of the semidefinite relaxations~\eqref{eq:prelaxed}. These conditions are initially a certificate for global optimality~\cite[theorem 7.4 and 7.5]{Las2015} and hold generically for a polynomial optimization problem~\cite[theorem 7.6]{Las2015}.
\end{rem}

\section{Distance to the transversely isotropic stratum of the symmetric second-order tensor}
\label{sec:isoT}

Let $\VV=\Sym^{2}(\RR^3)$ be the vector space of symmetric second-order tensors, endowed with the natural action $\rho_{2}(g)\ba =g\, \ba\, g^T$, for $\ba\in \Sym^{2}(\RR^3)$, $g\in G=\SO(3)$. Let $\bq=(\delta_{ij})$ be the \emph{Euclidean metric},
\begin{equation*}
  \ba'=\ba-\frac{1}{3} \tr(\ba)\, \bq ,
\end{equation*}
be the \emph{traceless part}  of $\ba$, $\lc$ be the \emph{Levi-Civita} tensor. We denote by $(\,)^{s}$ the \emph{total symmetrization} of a tensor. The generalized cross-product of symmetric tensors is defined as~\cite{OKDD2021}
\begin{equation}
  \ba \times \bb:=-(\ba\cdot \lc \cdot \bb)^s ,
  \qquad
  \left(i.e., \; (\ba\times\bb)_{ijk}:=-(a_{il}\varepsilon_{ljs} b_{sk})^{s} \right).
\end{equation}
For $\ba, \bb \in\Sym^{2}(\RR^3)$, it is a totally symmetric third-order tensor with 10 independent components.

There are three isotropy classes for the symmetric second-order tensors $(\Sym^{2}(\RR^3), \SO(3))$:
\begin{itemize}
  \item $[\DD_{2}]$ (orthotropy), if $\ba$ has three distinct eigenvalues,
  \item $[\OO(2)]$ (transverse isotropy, characterized by the polynomial equation $\ba^{2} \times \ba =0$~\cite[Lemma 8.1]{OKDD2021}), if $\ba$ has two distinct eigenvalues.
  \item $[\SO(3)]$ (isotropy, characterized by the linear equation
        $\ba'= 0$), if $\ba$ has three equal eigenvalues.
\end{itemize}

We illustrate through this first example the accuracy of Lasserre's polynomial optimization method to compute the \emph{distance to an isotropy stratum}. We shall obtain by this way the distance $\Delta(\ba_{0}, \cstrata{\OO(2)})$ of the orthotropic second-order tensor
\begin{equation*}
  \ba_{0}=
  \begin{pmatrix}
    -7 & 4  & -4 \\
    4  & 5  & -2 \\
    -4 & -2 & 5
  \end{pmatrix}
\end{equation*}
and compare the numerical results obtained with the algebraic solution derived in~\cite{ADKD2021}:
\begin{equation}\label{eq:exactDelta}
  \Delta(\ba_{0}, \cstrata{\OO(2)})^{2}=\norm{\ba_{0}-\ba^{**}}^{2}=18,
\end{equation}
where
\begin{equation}\label{eq:exactastar}
  \ba^{**}=\frac{1}{6} \left(
  \begin{array}{ccc}
      -44 & 20 & -20 \\
      20  & 31 & 5   \\
      -20 & 5  & 31  \\
    \end{array}
  \right)
  \approx
  \left(
  \begin{array}{ccc}
      -7.333333 & 3.333333  & -3.333333 \\
      3.333333  & 5.166667  & 0.8333333 \\
      -3.333333 & 0.8333333 & 5.166667  \\
    \end{array}
  \right).
\end{equation}

The numerical problem is first reduced to the following polynomial optimization problem
\begin{equation*}
  \min_{\ba \in K}\ \norm{\ba_{0}-\ba}^{2} ,
\end{equation*}
where
\begin{equation*}
  K=\set{\ba;\  \ba^{2} \times \ba= 0}.
\end{equation*}
Then, in order to properly apply the algorithm described in \autoref{sec:algo}, we need to ensure the Archimedean property (definition~\ref{def:archimedean_prop}), and for that we use lemma~\ref{lem:coercivity}. Therefore, to the 10 equations $\ba^{2} \times \ba= 0$, we add the inequality $c- \norm{\ba_{0}-\bb}^{2}\geq 0$ where we take
\begin{equation*}
  \bb=
  \begin{pmatrix}
    1 & 0 & 0 \\ 0 & 1 & 0 \\ 0 & 0 & -2
  \end{pmatrix}
  \in K,
\end{equation*}
and choose accordingly $c=300$. GloptiPoly then computes the approximation
\begin{equation}\label{eq:pbaKtilde}
  \norm{\ba_{0}-\ba^{*}}^{2}
\end{equation}
of the minimum $\Delta(\ba_{0}, \cstrata{\OO(2)})^{2} = \min_{\ba\in \widetilde{K}} \norm{\ba_{0}-\ba}^{2}$, where
\begin{equation}\label{eq:pbaKtildeset}
  \widetilde{K}=\set{\ba;\ \ba^{2} \times \ba= 0, c- \norm{\ba_{0}-\ba}^{2}\geq 0}.
\end{equation}

The optimal result computed in 1.2 seconds on a standard PC,
\begin{equation*}
  \Delta(\ba_{0}, \cstrata{\OO(2)})^{2}\approx 18.000007,
  \qquad
  \ba^{**}\approx \ba^{*}=
  \begin{pmatrix}
    -7.33299 & 3.33348 & -3.33348 \\
    3.33348  & 5.16651 & 0.83343  \\
    -3.33348 & 0.83343 & 5.16651  \\
  \end{pmatrix},
\end{equation*}
is close to the exact solution~\eqref{eq:exactDelta}--\eqref{eq:exactastar}, with the constraints accurately satisfied:
\begin{equation*}
  \frac{\max_{i} \abs{g_{i}(\ba^{*})}}{\norm{\ba_{0}}^{3}}=  \frac{\max_{p,q,r} \abs{(\ba^{*\, 2}\times \ba^{*})_{pqr}}}{\norm{\ba_{0}}^{3}}=  5.696 \, 10^{-9}.
\end{equation*}

For different values $c\in [202, 450]$, one gets
\begin{equation*}
  18.000007 \leq \Delta(\ba_{0}, \strata{\OO(2)})^{2}\leq 18.00006,
\end{equation*}
with the GloptiPoly convergence obtained for the first degree of relaxation $d=d_{0}=2$. The value chosen for $c$ affects the numerical solution. In fact, by increasing $c$ we get closer to the true minimum (=18), but the convergence is lost for $c\geq500$ (with a GloptiPoly status $\xi=0$ at the first relaxation).

\begin{rem}
  The transversely isotropic closed stratum $\cstrata{\OO(2)}$ can also be characterized by a single  scalar equation of degree 6,
  \begin{equation}\label{eq:gdea}
    g(\ba)=12\norm{\ba^{2}\times \ba}^{2}=\left(\tr (\ba^{\prime \,2})\right)^{3}-6 \left(\tr (\ba^{\prime 3})\right)^{2}=0,
  \end{equation}
  with $\ba'$, the traceless part of $\ba$. However, there is no finite convergence of the associated relaxation problem, since $\grad_{\ba} g(\ba)= \frac{6g(\ba)}{\tr (\ba^{\prime\,2})}\ba'=0$ when $g(\ba)=0$. In particular, the independence of the gradients of the constraint functions at the minimum (first order KKT sufficient condition mentioned in~\cite[theorem 6.5, theorem 7.2]{Las2015}, see remark~\ref{rem:global_optimality}) is not satisfied.
\end{rem}

The present example illustrates the strong dependence of the GloptiPoly convergence issue on the characterization of the isotropy classes. Indeed, convergence is obtained for the covariant characterization $\ba^{2} \times \ba=0$, but not for the invariant characterization~\eqref{eq:gdea}.

\section{Distance to cubic elasticity isotropy stratum}
\label{sec:ela}

In this section, we compute the distance of an experimental elasticity tensor $\bE_{0}$ to the cubic isotropy stratum $\cstrata{\octa}$, and determine the associated minimizer $\bE^{*}$. The distance to an isotropy stratum problem has been widely addressed in the Continuum Mechanics literature, by solving it in terms of an unknown rotation (either parameterized by Euler angles~\cite{FBG1996,FGB1998}, or by a unit quaternion~\cite{Del2005,KS2008,KS2009}). Here, we use the characterization of the (cubic) isotropy stratum by means of at most quadratic covariants in order to formulate such a distance problem as a quadratic polynomial optimization problem. This makes us able to apply Lasserre's method, and to show that using GloptiPoly allows to compute an accurate solution of this non trivial example.

\subsection{Formulation of the distance problem as a polynomial optimization problem}

Let
\begin{equation*}
  \VV=\Ela=\set{\bE\in \otimes^{4}\RR^{3},\, E_{ijkl}=E_{jikl}=E_{klij}}
  \qquad
  (\dim \Ela=21)
\end{equation*}
be the set of elasticity tensors $\bE: \Sym^{2}(\RR^3) \to \Sym^{2}(\RR^3)$, introduced in example~\ref{ex:Ela}, and $G=\SO(3)$.
An \emph{elasticity tensor} $\bE\in \Ela$ can be represented by a $6\times 6$ symmetric matrix, in Voigt notation,
\begin{equation}
  \label{eq:Voigt}
  [\bE]=\begin{pmatrix}
    E_{1111} & E_{1122} & E_{1133} & E_{1123} & E_{1113} & E_{1112} \\
    E_{1122} & E_{2222} & E_{2233} & E_{2223} & E_{1223} & E_{1222} \\
    E_{1133} & E_{2233} & E_{3333} & E_{2333} & E_{1333} & E_{1233} \\
    E_{1123} & E_{2223} & E_{2333} & E_{2323} & E_{2331} & E_{2312} \\
    E_{1113} & E_{1223} & E_{1333} & E_{2331} & E_{1313} & E_{3112} \\
    E_{1112} & E_{1222} & E_{1233} & E_{2312} & E_{3112} & E_{1212}
  \end{pmatrix}.
\end{equation}

The vector space $\Ela$ decomposes into a direct sum of $\SO(3)$-irreducible subspaces (so-called harmonic decomposition~\cite{Bac1970,Spe1970})
\begin{equation*}
  \Ela=\HH^0(\RR^3)\oplus \HH^0(\RR^3) \oplus \HH^{2}(\RR^3) \oplus \HH^{2}(\RR^3) \oplus \HH^{4}(\RR^3),
\end{equation*}
where $\HH^{n}(\RR^3)$ denotes the space of harmonic tensors of order $n$ ($\dim \HH^{n}(\RR^3)=2n+1$).
Letting
\begin{equation*}
  \bd=\tr_{12} \bE, \qquad \bv=\tr_{13} \bE,
\end{equation*}
the harmonic decomposition of $\bE\in \Ela$ can be expressed as (see \autoref{sec:Elast-Harm-Dec} for explicit formulas)
\begin{equation*}
  \bE=(\alpha,\beta,\bd',\bv',\bH) ,
\end{equation*}
with $\alpha=\tr \bd, \beta=\tr \bv\in \HH^0(\RR^3)$ the scalar (isotropic) components of $\bE$, with $\bd', \bv'\in\HH^{2}(\RR^3)$ its second-order harmonic components (the traceless  parts of $\bd$ and $\bv$),
and $\bH\in \HH^{4}(\RR^3)$ its fourth-order harmonic component. The squared Euclidean norm of $\bE$ is then
\begin{equation}\label{eq:normE2}
  \norm{\bE}^{2} = 5 \alpha^{2} + 4 \beta ^{2} + \frac{2}{21} \norm{\bd^{\prime} + 2 \bv^{\prime}}^{2} + \frac{4}{3}\norm{\bd^{\prime} - \bv^{\prime}}^{2}+ \norm{\bH}^{2}.
\end{equation}

We consider the triclinic experimental elasticity tensor $\bE_{0}$ representing the Nickel-based aeronautics single crystal superalloy (of CMSX-4 type), measured in~\cite{FGB1998}. In Voigt notation:
\begin{equation}\label{eq:E0}
  [\bE_{0}]=
  \begin{pmatrix}
    243 & 136 & 135 & 22 & 52 & -17 \\ 136 & 239 & 137 & -28 & 11 & 16 \\ 135 & 137 & 233 & 29 & -49 & 3 \\ 22 & -28 & 29 & 133 & -10 & -4 \\ 52 & 11 & -49 & -10 & 119 & -2 \\ -17 & 16 & 3 & -4 & -2 & 130
  \end{pmatrix}
  \;\text{ GPa}.
\end{equation}
This material has an expected symmetry, namely the cubic symmetry $[\octa]$, deduced from its cubic microstructure (see figure~\ref{fig:microstructure}).

\begin{figure}[h!]
  \includegraphics[width=6.5cm]{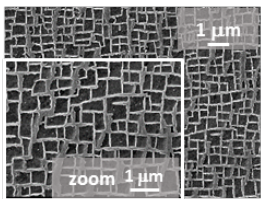}
  \caption{Cubic microstructure of CMSX-4 Ni-based single crystal superalloy~\cite{MDC2019}.}
  \label{fig:microstructure}
\end{figure}

We then aim a computing of
\begin{equation}\label{eq:dist-ela1}
  \Delta(\bE_{0},\cstrata{\octa})^{2}=\min_{\bE\in \overline{\Sigma}_{\octa}} \norm{\bE_{0}-\bE}^{2}.
\end{equation}
This optimization problem has 21 variables $E_{ijkl}$.
To set it as a polynomial optimization problem, we take advantage of the fact that the cubic elasticity stratum is an algebraic set, characterized by explicit polynomial equations.

\begin{thm}[Olive et al~\cite{OKDD2021}]\label{thm:Ecubic}
  Let $\bE=(\alpha, \beta, \bd', \bv', \bH)\in \Ela$ be an elasticity tensor,
  \begin{equation*}
    \bd_{2}=\bH\3dots \bH
    \qquad
    \left(
    \textit{i.e.}, \; (\bd_{2})_{ij}= H_{ipqr} H_{pqrj}
    \right)
  \end{equation*}
  and
  $\bd_{2}'=\bd_{2}-\frac{1}{3} \tr(\bd_{2})\, \bq$ be second-order covariants of $\bE$. Then
  $\bE\in \cstrata{\octa}$ (is at least cubic) if and only if
  \begin{equation*}
    \bv'=\bd'=0
    \quad\textrm{and} \quad
    \bd_{2}'=0,
  \end{equation*}
  and $\bE\in \strata{\octa}$ (is cubic) if and only if furthermore $\bH\neq 0$.
\end{thm}

We have then the following result.

\begin{thm}\label{thm:equiPBEla}
  Let $\bE=(\alpha, \beta, \bd', \bv', \bH)$ and $\bE_{0}=(\alpha_{0}, \beta_{0}, \bd'_{0}, \bv'_{0}, \bH_{0})$ be two elasticity tensors. The 21-dimensional minimization problem~\eqref{eq:dist-ela1} is equivalent to the 9-dimensional polynomial optimization problem
  \begin{equation*}
    \min_{\bd_{2}' =0} \norm{\bH_{0}-\bH}^{2},
  \end{equation*}
  with $\bE=(\alpha_{0}, \beta_{0}, 0, 0, \bH)$, and
  \begin{equation}\label{eq:normE0moinsE2}
    \min_{\bE\in \overline{\Sigma}_{\octa}}\norm{\bE_{0}-\bE}^{2} = \frac{2}{21} \norm{\bd_{0}^{\prime}+2 \bv_{0}^{\prime}}^{2} +\frac{4}{3}\norm{\bd_{0}^{\prime}- \bv_{0}^{\prime}}^{2}+\min_{\bd_{2}'=0} \norm{\bH_{0}-\bH}^{2}.
  \end{equation}
\end{thm}

\begin{proof}
  The squared distance function in~\eqref{eq:dist-ela1} is
  \begin{equation*}
    \norm{\bE_{0}-\bE}^{2}=\norm{(\alpha_{0}-\alpha,\beta_{0}-\beta,\bd'_{0}-\bd',\bv'_{0}-\bv',\bH_{0}-\bH)}^{2}.
  \end{equation*}
  It can be expressed as
  \begin{equation*}
    \norm{\bE_{0}-\bE}^{2} =5 (\alpha_{0}-\alpha)^{2}+4 (\beta_{0}-\beta)^{2}+  \frac{2}{21} \norm{\bd_{0}^{\prime}+2 \bv_{0}^{\prime}-(\bd^{\prime}+2 \bv^{\prime})}^{2} +\frac{4}{3}\norm{\bd_{0}^{\prime}- \bv_{0}^{\prime}-(\bd^{\prime}- \bv^{\prime})}^{2}+ \norm{\bH_{0}-\bH}^{2},
  \end{equation*}
  by using~\eqref{eq:normE2}. By theorem~\ref{thm:Ecubic},  taking $\alpha=\alpha_{0}$, $\beta=\beta_{0}$, $\bd'=0$ and $\bv'=0$,  we get~\eqref{eq:normE0moinsE2}, and the polynomial optimization problem~\eqref{eq:dist-ela1} is reduced to the following problem in only 9 variables (the components of $\bH\in \HH^{4}(\RR^3)$, $\dim  \HH^{4}(\RR^3)=9$) instead of 21,
  \begin{equation*}
    \min_{\bd_{2}'=0} \norm{\bH_{0}-\bH}^{2}.
  \end{equation*}
\end{proof}

\begin{rem}
  An elasticity tensor $\bE\in \Ela$ corresponds to a quadratic elastic energy density, which must be positive semidefinite. This condition can be characterized using $\SO(3)$-invariant polynomial inequalities on $\bE$, and thus added to the set of polynomial constraints, if necessary, using the following fact. Given a symmetric real $n\times n$ matrix $A$, we get
  \begin{equation*}
    \text{$A$ is positive semidefinite} \iff  \forall i \in \set{1,\ldots,n}, \;\sigma_i \ge 0 ,
  \end{equation*}
  where $\sigma_{1},\ldots,\sigma_n$ denote the elementary symmetric polynomials in the eigenvalues $\lambda_i$ of $A$. Indeed, if $\lambda_i \ge 0$ for all $i \in \set{1,\ldots,n}$, then $\sigma_i \ge 0$ for all $i \in \set{1,\ldots,n}$. Conversely,
  assume that  $\sigma_i \ge 0$ for all $i \in \set{1,\ldots,n}$.
  Then, the polynomial
  \begin{equation*}
    p :=(X+\lambda_{1})(X+\lambda_{2})\ldots(X+\lambda_n) =X^{n}+\sigma_{1}X^{n-1}+\ldots+\sigma_{n-1} X +\sigma_n,
  \end{equation*}
  satisfies
  \begin{equation*}
    p(x)\geq x^{n} > 0, \qquad \forall x>0.
  \end{equation*}
  Hence, the (real) roots of $p$, namely $-\lambda_{1},\ldots,-\lambda_n$, belong to $]-\infty, 0]$, and thus $\lambda_i \ge 0$ for all $i \in \set{1,\ldots,n}$.
\end{rem}

In practice, an experimental tensor $\bE_{0}$ is most often measured as semidefinite and  the tensor the closest to $\bE_{0}$ computed as semidefinite, so, here we do not add the semidefiniteness constraint $\bE \geq 0$ to our optimization problem.

\subsection{Resolution by Lasserre's method}

A fourth order harmonic tensor $\bH\in \HH^{4}(\RR^3)$ is represented by the following real matrix (in Voigt notation)
\begin{equation}\label{eq:HVoigt}
  [\bH]=\begin{pmatrix}
    \Lambda_{2}+\Lambda_{3} & -\Lambda_{3} & -\Lambda_{2} & -X_{1} & Y_{1}+Y_{2} & -Z_{2} \\ -\Lambda_{3} & \Lambda_{3}+\Lambda_{1} & -\Lambda_{1} & -X_{2} & - Y_{1} & Z_{1}+Z_{2} \\ -\Lambda_{2} & -\Lambda_{1} & \Lambda_{1}+\Lambda_{2} & X_{1}+X_{2} & -Y_{2} & -Z_{1} \\ -X_{1} & -X_{2} & X_{1}+X_{2} & -\Lambda_{1} & -Z_{1} & -Y_{1} \\ Y_{1}+Y_{2} & - Y_{1} & -Y_{2} & -Z_{1} & -\Lambda_{2} & -X_{1} \\ -Z_{2} & Z_{1}+Z_{2} & -Z_{1} & -Y_{1} & -X_{1} & -\Lambda_{3}
  \end{pmatrix}.
\end{equation}
In practice we set
\begin{equation*}
  \xx = (\Lambda_{1}, \Lambda_{2}, \Lambda_{3}, X_{1}, X_{2}, Y_{1}, Y_{2}, Z_{1}, Z_{2}),
\end{equation*}
and GloptiPoly computes the approximation
\begin{equation}\label{eq:minH0moinsH}
  \norm{\bH_{0}-\bH^{*}}^{2}
\end{equation}
of the minimum $\Delta(\bH_{0}, \cstrata{\octa})^{2}=\min_{\xx \in \widetilde{K}} f(\xx)$, where $f(\xx)=\norm{\bH_{0}-\bH}^{2}$ and
\begin{equation*}
  \widetilde{K} = \set{\xx;\  \bd_{2}'=0, \,c-f(\xx) \ge 0},
\end{equation*}
with $c=58000>f(0)$ to ensure the Archimedean property on the set of constraints. The five quadratic scalar equations $(\bd_{2}')_{ij} = 0$ are detailed in \autoref{sec:d2-prime}. For $\bE_{0}$ given by~\eqref{eq:E0}, we have
\begin{align*}
  f(\xx) & =  540\Lambda_{2} + 620\Lambda_{3} - 88X_{1} + 668\Lambda_{1} - 264Z_{1} - 264Z_{2} - 456X_{2} + 8Z_{2}^{2}
  +8\Lambda_{1}^{2}
  \\
         & \quad +8\Lambda_{2}^{2}+8\Lambda_{3}^{2}+8Y_{2}^{2}+16X_{1}^{2}+8X_{2}^{2}+16Y_{1}^{2}+16Z_{1}^{2}-392Y_{1}-808Y_{2}
  \\
         & \quad  + 8Z_{1}Z_{2} + 2\Lambda_{1}\Lambda_{2} + 2\Lambda_{2}\Lambda_{3} + 8X_{1}X_{2} + 8Y_{1}Y_{2} + 2\Lambda_{3}\Lambda_{1} + \frac{2026042}{35}.
\end{align*}
We obtain the result at the first relaxation order $d=d_{0}=1$ with GloptiPoly status $\xi = +1$ and value
\begin{equation*}
  \min_{\xx \in \widetilde{K}} f(\xx) \approx f(\xx^{*})=2530.474727 \; \text{GPa}^{2},
\end{equation*}
The computation time is of 0.9 seconds. The computed minimizer is
\begin{multline*}
  \xx^{*} = (-36.401489,-20.227012,-38.908985,-6.396664,27.780748,\\-2.277546,44.251364,-4.557344,21.161507).
\end{multline*}
By~\eqref{eq:HVoigt}, it corresponds to the fourth-order harmonic tensor $\bH^{*}$ solution of~\eqref{eq:minH0moinsH}.

We get, by theorem~\ref{thm:equiPBEla}, $\bE^{*}=(\alpha_{0}, \beta_{0}, 0, 0, \bH^{*})$, \emph{i.e.},
\begin{equation*}
  \bE^{*}=\frac{1}{15}\left(\alpha_{0}+2 \beta_{0}\right)  \bq \odot 
  \bq+\frac{1}{6}\left(\alpha_{0}-\beta_{0}\right) \bq\underset{(2,2)}{\otimes} \bq +
  \bH^{*} ,
\end{equation*}
with $\odot$ the symmetric tensor product (see \autoref{sec:Elast-Harm-Dec}). The elasticity tensor $\bE^{*}$ is cubic (and not isotropic) since $\bH^{*}\neq 0$.

Finally, the computed cubic tensor $\bE^{*}\in \Sigma_{[\octa]}$ the closest to $\bE_{0}$ is, in Voigt notation,
\begin{equation*}
  [\bE^{*}]=
  \begin{pmatrix}
    240.130669 & 144.442318 & 125.760345 & 6.39666    & 41.97381   & -21.161507 \\
    144.442318 & 223.956191 & 141.934823 & -27.780748 & 2.277546   & 16.604162  \\
    125.760345 & 141.934823 & 242.638164 & 21.384084  & -44.251364 & 4.557344   \\
    6.39666    & -27.780748 & 21.384084  & 133.268156 & 4.557344   & 2.277546   \\
    41.973817  & 2.277546   & -44.251364 & 4.557344   & 117.093678 & 6.39666    \\
    -21.161507 & 16.604162  & 4.557344   & 2.277546   & 6.39666    & 135.775651
  \end{pmatrix}
  \text{ GPa}.
\end{equation*}
It corresponds to $\Delta(\bE_{0}, \cstrata{\octa})\approx \norm{\bE_{0}-\bE^{*}}=74.131148$ GPa and to the relative distance to cubic symmetry
\begin{equation*}
  \frac{\norm{\bE_{0}-\bE^{*}}}{\norm{\bE_{0}}}=0.103910,
\end{equation*}
slightly better than the solution computed in~\cite{FGB1998} using a parameterization by Euler angles together with a simplex minimization method. Note that the constraint $\bd_{2}'=0$ is satisfied accurately, since
\begin{equation*}
  \frac{\bd_{2}'}{\norm{\bH_{0}}^{2}} = 10^{-6}
  \begin{pmatrix}
    -4.097         & -2.8\ 10^{-6} & -4.9 \ 10^{-6} \\
    -2.8\ 10^{-6}  & -4.455        & 4.2 \ 10^{-6}  \\
    -4.9 \ 10^{-6} & 4.2 \ 10^{-6} & -8.552
  \end{pmatrix}
  \approx \mathbf{0}.
\end{equation*}

One can choose other values for $c$ satisfying $c>f(\yy)$ for some $\yy\in K$. The GloptiPoly solution varies slightly as $c$ runs the interval $[58000, 61000]$, with a computation time of 0.9 seconds for $c=58000$,  of 0.8 seconds for $c=60000$, and of 0.1 seconds for $c=61000$. Outside from this narrow interval, the GloptiPoly convergence is lost (Gloptipoly status $\xi = 0$).

\begin{rem}
  The computation time is lower for this quadratic optimization problem (with 9 variables) than for the degree 3 polynomial optimization problem of \autoref{sec:isoT} (with 6 variables).
\end{rem}

\section{Distance to cubic piezoelectricity isotropy stratum}
\label{sec:piezo}

In this final section, we apply Lasserre's polynomial optimization method to compute the distance $\Delta(\be_{0},\cstrata{\octa^{-}})$ of a raw piezoelectricity third-order tensor\footnote{relating induced polarization in a dielectric material to the strain tensor.} $\be_{0}$ to the cubic piezoelectricity stratum $\cstrata{\octa^{-}}$. This problem seems to have  never been addressed before. It is important for the design of dielectric materials, since for instance the piezolectricity behavior strongly depends on the crystal primitive cell symmetry.

\subsection{Formulation of the distance problem as a polynomial optimization problem}

According to the three-dimensional piezoelectricity framework~\cite{EM1990,GW2002}, we denote by
\begin{equation*}
  \VV=\Piez=\set{\be\in \otimes^{3}\RR^{3},\, \re_{ijk}=\re_{ikj}}
  \qquad
  (\dim \Piez=18),
\end{equation*}
the vector space of piezoelectricity tensors $\be:\Sym^{2}(\RR^3) \to \RR^3$ (see example~\ref{ex:Piez}), and set $G=\OO(3)$.
A \emph{piezoelectricity tensor} $\be\in \Piez$ can be represented by a $3 \times 6$ matrix, in so-called Voigt representation,
\begin{equation*}
  [\be]=
  \begin{pmatrix}
    \re_{111} & \re_{122} & \re_{133} & \re_{123} & \re_{113} & \re_{112} \\
    \re_{211} & \re_{222} & \re_{233} & \re_{223} & \re_{213} & \re_{212} \\
    \re_{311} & \re_{322} & \re_{333} & \re_{323} & \re_{313} & \re_{312}
  \end{pmatrix}.
\end{equation*}
The vector space $\Piez$ decomposes into a direct sum of $\OO(3)$-irreducible subspaces (so-called harmonic decomposition~\cite{Spe1970})
\begin{equation*}
  \Piez=\HH^{1}(\RR^3)\oplus \HH^{1}(\RR^3) \oplus \HH^{2 \sharp}(\RR^3) \oplus \HH^{3}(\RR^3).
\end{equation*}
The notation $\HH^{n}(\RR^3)$ still refers to the vector space of $n$-th order harmonic tensors endowed with the standard $\OO(3)$-representation $\rho_{n}$, while $\HH^{n \sharp}(\RR^3)$ refers to the same vector space endowed with the \emph{twisted} $\OO(3)$-representation $\hat \rho_{n}$, such that $\hat \rho_{n}(g)=(\det g) \, \rho_{n}(g)$.
One has
\begin{equation*}
  \be=(\vv, \ww, \ba, \bh)
\end{equation*}
with $\vv, \ww\in \HH^{1}(\RR^{3})$, $\ba\in  \HH^{2 \sharp}(\RR^{3})$ and $\bh \in \HH^{3}(\RR^{3})$.

Let $\odot$ be the symmetric tensor product and $\be^{s}\in \Sym^{3}(\RR^{3})$ denote the totally symmetric part of $\be$ (of components $(\be^{s})_{ijk}=\frac{1}{3}(\re_{ijk}+\re_{jik}+\re_{kji})$). Any piezoelectricity tensor $\be\in \Piez$ can be decomposed as the sum
\begin{equation*}
  \be=\bg+\bh
\end{equation*}
where
\begin{equation}\label{eq:Hpiezoformule}
  \bh:=\be^{s}-\frac{3}{5} \bq \odot \tr(\be^{s})
  \in \HH^{3}(\RR^3) ,
\end{equation}
is the \emph{leading harmonic part} of $\be$, and
\begin{equation*}
  \bg:=\be- \bh=(\vv, \ww, \ba),
\end{equation*}
is orthogonal to $\bh$ (\emph{i.e.}, $\langle \bg, \bh \rangle=g_{ijk} h_{ijk}=0$).

\begin{rem}
  The third-order tensors $\bg=\bg(\be)$ and $\bh=\bh(\be)$ are linear covariants of $\be$.
\end{rem}

The squared Euclidean norm of $\be$ is then
\begin{equation}\label{eq:normepiez2}
  \norm{\be}^{2}= \re_{ijk} \re_{ijk} = \norm{\bg}^{2}+\norm{\bh}^{2}.
\end{equation}

We will first consider the following raw (triclinic) piezoelectricity tensor $\be_{0}$ for pure wurtzite AlN (aluminum nitride, $x=0$), of Voigt representation,
\begin{equation}\label{eq:PPiez}
  [\be_{0}]=\left(
  \begin{array}{cccccc}
      0       & 0       & -0.0505 & -0.0394 & -0.2854 & -0.0637 \\
      -0.0637 & -0.0042 & 0.0332  & -0.2818 & -0.0058 & 0.0185  \\
      -0.5807 & -0.5822 & 1.4607  & 0.0022  & 0.0002  & 0.0043  \\
    \end{array}
  \right)
  \;\text{C/m}^{2} ,
\end{equation}
in Coulomb per square meter, computed by Density Functional Theory (DFT), using ab-initio simulations, by Manna and coworkers~\cite[Fig.3]{MTG2018}. We will also consider wurtzite alloys $\text{Cr}_{x}\text{Al}_{1-x}\text{N}$ and the associated raw  piezoelectricity tensors $\be_{0}^{x}$ (given in the \autoref{sec:Piez-Mannna} for chromium concentrations $0\leq x \leq 0.25$). Note that pure rocksalt $\text{CrN}$ corresponds to a Cr-concentration $x=1$, and that the value $x = 0.25$ is the so-called wurzite to rocksalt phase transition point~\cite{MMRFS2008}.

We aim at computing by polynomial optimization
\begin{equation}\label{eq:dist-piez1}
  \Delta(\be_{0},\cstrata{\octa^{-}})^{2}=\min_{\be\in \overline{\Sigma}_{\octa^{-}}} \norm{\be_{0}-\be}^{2},
\end{equation}
and $\be^{*}\in\strata{\octa^{-}}$ the closest to $\be_{0}$. In order to succeed, we first have to characterize the cubic piezoelectricity stratum $\cstrata{\octa^{-}}$ by polynomial equations (a proof of the following theorem is provided in \autoref{sec:proof-d2-prime}).

\begin{thm}\label{thm:Pcubic}
  Let $\be=\bg+ \bh\in \Piez$ be a piezoelectricity tensor, with $\bh\in \HH^{3}(\RR^{3})$ its leading harmonic part, let
  \begin{equation*}
    \bd_{2}=\bh : \bh
    \qquad \left(\emph{i.e.}, \, (\bd_{2})_{ij}=h_{ikl}h_{klj} \right),
  \end{equation*}
  and $\bd_{2}'=\bd_{2}-\frac{1}{3} \tr(\bd_{2})\, \bq$ be second-order covariants of $\be$. Then $\be\in \cstrata{\octa^{-}}$ (is at least cubic) if and only if
  \begin{equation*}
    \bg=0 \quad \text{and} \quad \bd_{2}'=0,
  \end{equation*}
  and $\be\in \strata{\octa^{-}}$ (is cubic) if and only if furthermore $\bh\neq 0$.
\end{thm}

With the same proof as for theorem~\ref{thm:equiPBEla}, we have the following result.

\begin{thm}\label{thm:equiPBPiez}
  Let $\be=\bg+ \bh$ and $\be_{0}=\bg_{0}+ \bh_{0}$ be two piezoelectricity tensors, with $\bh$ and $\bh_{0}$ their leading harmonic parts.
  The 15-dimensional minimization problem~\eqref{eq:dist-piez1} is equivalent to the 7-dimensional polynomial optimization problem
  \begin{equation*}
    \min_{\bd_{2}' =0} \norm{\bh_{0}-\bh}^{2},
  \end{equation*}
  with $\be= \bh$, and
  \begin{equation*}
    \min_{\be\in \overline{\Sigma}_{\octa^{-}}}\norm{\be_{0}-\be}^{2} =  \norm{\bg_{0}}^{2} +\min_{\bd_{2}'=0} \norm{\bh_{0}-\bh}^{2}.
  \end{equation*}
\end{thm}

\subsection{Resolution by Lasserre's method}

A third order harmonic tensor $\bh\in \HH^{3}(\RR^{3})$ has seven independent components and is represented by the following real matrix (in Voigt notation)
\begin{equation}\label{eq:Hpiezomatrice}
  [\bh]=
  \begin{pmatrix}
    h_{111}          & h_{122} & -h_{111}-h_{122} & h_{123}          & -h_{223}-h_{333} & h_{112} \\
    h_{112}          & h_{222} & -h_{112}-h_{222} & h_{223}          & h_{123}          & h_{122} \\
    -h_{223}-h_{333} & h_{223} & h_{333}          & -H_{112}-h_{222} & -h_{111}-h_{122} & h_{123}
  \end{pmatrix}
\end{equation}
The traceless second order tensor $\bd_{2}'$ has five independent components $(\bd_{2}')_{ij}$ detailed in \autoref{sec:d2-prime}.

We set
\begin{equation*}
  \xx=( h_{111},\, h_{112}, \, h_{122}, \, h_{123}, \, h_{222}, \, h_{223}, \,  h_{333}).
\end{equation*}
GloptiPoly computes the approximation
\begin{equation}\label{eq:minh0moinshpiez}
  \norm{\bh_{0}-\bh^{*}}^{2}
\end{equation}
of the minimum $\Delta(\be_{0},\cstrata{\octa^{-}})^{2} = \min_{\xx \in \widetilde{K}} f(\xx)$, where  $f(\xx)=\norm{\bh_{0}-\bh}^{2}$ and
\begin{equation*}
  \widetilde{K}=\set{ \xx;\  \bd_{2}'=0, \, c - f(\xx)\geq 0},
  \qquad c=3.
\end{equation*}
For $\be_{0}$ given by~\eqref{eq:PPiez}, we have (in C$^{2}$/m$^{4}$)
\begin{align*}
  f(\xx) & = 6h_{111}h_{122}+6h_{223}h_{333}+6h_{112}h_{222}+4h_{111}^{2}+6h_{112}^{2}+6h_{122}^{2}+4h_{222}^{2}
  \\
         & \quad +6h_{223}^{2}+4h_{333}^{2}+6h_{123}^{2}-0.1002 h_{111}-0.1742 h_{122}+0.1636 h_{123}
  \\
         & \quad -0.0114 h_{223}-5.2244 h_{333}+0.4574 h_{112}+0.0836 h_{222}+2.7367.
\end{align*}

\begin{rem}
  We take $c = 3 > f(0)$ to ensure the Archimedean property, but in the present case the convergence status $\xi$ does not seem to depend on $c$. Dropping the condition $c - f(\xx)\geq 0$  in $\tilde K$ also leads to an accurate computed optimum.
\end{rem}

We obtain the result $\min_{\xx \in \widetilde{K}} f(\xx) \approx f(\xx^{*})=1.060855$ C$^{2}$/m$^{4}$ at the first GloptiPoly relaxation $d=d_{0}=1$ (with convergence status $\xi=+1$ and for a computation time of 0.8 seconds). The components of the computed minimizer $\bh^{*}$ are (in C/m$^{2}$):
\begin{multline*}
  h_{111} = -0.075476 ,~h_{112}=-0.426450, ~h_{122} = 0.088998,~ h_{123} = -0.005937,
  \\
  ~h_{222}=0.412070, ~h_{223} = -0.308913, ~h_{333}=0.609783.
\end{multline*}

By theorem~\ref{thm:equiPBPiez}, the computed cubic tensor $\be^{*}\in \Sigma_{[\octa^{-}]}$ the closest to $\be_{0}$ is simply $\be^{*}=\bh^{*}$. In Voigt notation,
\begin{equation*}
  [\be^{*}]=
  \begin{pmatrix}
    -0.075476 & 0.088998  & -0.013521 & -0.005937 & -0.300870 & -0.426450 \\
    -0.426450 & 0.412070  & 0.0143797 & -0.308913 & -0.005937 & 0.088998  \\
    -0.300870 & -0.308913 & 0.609783  & 0.014379  & -0.013521 & -0.005937
  \end{pmatrix}
  \text{ C/m}^{2}
\end{equation*}
The distance and the relative distance to cubic piezoelectricity are finally
\begin{equation*}
  \Delta(\be_{0}, \cstrata{\octa^{-}})\approx \norm{\be_{0}-\be^{*}}=1.214681 \, \text{C/m}^{2},
  \qquad
  \frac{\norm{\be_{0}-\be^{*}}}{ \norm{\be_{0}}}=0.684256.
\end{equation*}
The results obtained for the raw piezoelectricity tensors $\be_{0}^{x}$ given in the \autoref{sec:Piez-Mannna} for wurzite $\text{Cr}_{x}\text{Al}_{1-x}\text{N}$, with different chromium concentrations, are summarized in \autoref{tab:res-wurzite}.

\begin{table}[h]
  \begin{center}
    \setlength{\arraycolsep}{1pt}
    \begin{tabular}{c|ccc}
      \toprule
      $x$     & $\Delta(\bE_{0}, \cstrata{\octa^{-}})$ & $\displaystyle\frac{\norm{\be_{0}-\be^{*}}}{ \norm{\be_{0}}}$ & Computation time (s) \\
      \midrule
      0 (AlN) & 1.214681                               & 0.684256                                                      & 0.7                  \\
      \midrule
      0.035   & 1.307327                               & 0.715295                                                      & 0.7                  \\
      \midrule
      0.07    & 1.364909                               & 0.729065                                                      & 0.8                  \\
      \midrule
      0.10    & 1.541726                               & 0.785604                                                      & 0.6                  \\
      \midrule
      0.13    & 1.542293                               & 0.758240                                                      & 1.0                  \\
      \midrule
      0.16    & 1.665883                               & 0.793355                                                      & 0.6                  \\
      \midrule
      0.19    & 1.852505                               & 0.813719                                                      & 0.7                  \\
      \midrule
      0.225   & 1.877377                               & 0.781094                                                      & 1.2                  \\
      \midrule
      0.255   & 1.944763                               & 0.752770                                                      & 0.7                  \\
      \bottomrule
    \end{tabular}
    \caption{Results for the raw piezoelectricity tensors of~\cite{MTG2018} for different Cr-concentrations $x$ (the distance $\Delta(\bE_{0}, \cstrata{\octa^{-}})\approx \norm{\be_{0}-\be^{*}}$ is in C/m$^{2}$).}
    \label{tab:res-wurzite}
  \end{center}
\end{table}

\begin{rem}
  The computation times are of the same order of magnitude as for the cubic elasticity case. For this quadratic optimization problem (with 7 variables) as well, they are lower than for the 6 variables but degree 3 optimization problem of \autoref{sec:isoT}.
\end{rem}

\section{Conclusion}

Some isotropy strata of tensorial representations of the orthogonal group are explicitly characterized by polynomial covariants. We have taken advantage of this fact to formulate the computation of the distance to these strata as a polynomial optimization problem. We have used the property that the isotropy classes for the representation of $\SO(3)$ on the vector space of elasticity tensors are in general semialgebraic. The present work shows the interest of the characterization of the isotropy classes by means of polynomial covariants, rather than by means of invariants. In particular, the covariant characterization of the cubic piezoelectricity symmetry stratum (theorem~\ref{thm:Pcubic}), which is the cornerstone of our methodology, is a new result.

We have then recalled Lasserre's method to solve polynomial optimization problems under semialgebraic constraints. Under the so-called Archimedean hypothesis, this approach consists in writing the initial problem as an infinite semidefinite program from which is constructed a sequence of relaxed semidefinite programs that converges to the desired global minimum. We have presented the corresponding algorithm implemented in the freeware GloptiPoly, in particular its stopping criterion.

We have applied this polynomial optimization method to compute the cubic tensor the closest to a raw (measured) constitutive tensor, both in continuum mechanics elasticity and piezoelectricity. We have considered the following examples
\begin{itemize}
  \item of an elasticity tensor measured by François and coworkers~\cite{FGB1998} for an aeronautics Nickel-based single crystal superalloy,
  \item of nine piezoelectricity tensors computed for wurtzite alloys using
        Density Functional Theory (DFT) and ab-initio simulations,
        by Manna and coworkers~\cite{MTG2018}.
\end{itemize}

In both cases, we took advantage of the distance being a coercive polynomial function to adapt the constraints so that they can satisfy the Archimedean condition, in order to ensure the convergence of the method to the desired minimum.

\appendix

\section{Explicit harmonic decomposition of an elasticity tensor}
\label{sec:Elast-Harm-Dec}

An elasticity tensor $\bE \in \Ela$ admits the following explicit harmonic decomposition~\cite{Bac1970}:
\begin{equation}\label{eq:harm_decomp}
  \bE=\bE^{iso}+\bq\underset{(4)}{\otimes} \ba+\bq\underset{(2,2)}{\otimes}\bb+\bH .
\end{equation}
where
\begin{equation}\label{eq:Eiso}
  \bE^{iso}= \frac{1}{15}\left(\alpha+2 \beta\right) \bq\underset{(4)}{\otimes} \bq+\frac{1}{6}\left(\alpha-\beta\right) \bq\underset{(2,2)}{\otimes} \bq,
\end{equation}
\begin{equation}\label{eq:alphabeta}
  \alpha=\tr(\bd), \qquad
  \beta=\tr(\bv),
\end{equation}
and
\begin{equation}\label{eq:ab}
  \ba=\frac{2}{7}(\bd'+2\bv'), \qquad
  \bb=2(\bd'-\bv').
\end{equation}
with $\bd'=\bd-\frac{1}{3}\tr(\bd)\bq$ and $\bv'=\bv-\frac{1}{3}\tr(\bv)\bq$ respectively the traceless parts of $\bd=\tr_{12}\bE$ and $\bv=\tr_{13}\bE$.

In~\eqref{eq:harm_decomp}, $\bq$ is the Euclidean canonical bilinear 2-form represented by the components $(\delta_{ij})$ in any orthonormal basis and the tensor products $\underset{(4)}{\otimes}$ and $\underset{(2,2)}{\otimes}$, between symmetric second-order tensors $\ba$, $\bb$, are defined as follows:
\begin{equation*}
  (\ba\underset{(4)}{\otimes} \bb)_{ijkl}= (\ba\odot\bb)_{ijkl}=\frac{1}{6} (a_{ij} b_{kl} + b_{ij}a_{kl} + a_{ik}b_{jl} + b_{ik}a_{jl} + a_{il}b_{jk} + b_{il}a_{jk}),
\end{equation*}
and
\begin{equation*}
  (\ba\underset{(2,2)}{\otimes}\bb)_{ijkl} = \frac{1}{6}
  (2a_{ij} b_{kl} + 2b_{ij}a_{kl}-a_{ik}b_{jl}-a_{il}b_{jk}- b_{ik}a_{jl}-b_{il}a_{jk}).
\end{equation*}
We have, for Euclidean norm,
\begin{equation*}
  \norm{\bE}^{2} =5 \alpha ^{2}+4 \beta ^{2}+  \frac{2}{21} \norm{\bd^{\prime}+2 \bv^{\prime}}^{2} +\frac{4}{3}\norm{\bd^{\prime}- \bv^{\prime}}^{2}+ \norm{\bH}^{2},
\end{equation*}

Using~\eqref{eq:alphabeta}--\eqref{eq:ab} we obtain for the experimental elasticity tensor $\bE_{0}$ (given by~\eqref{eq:E0}) the harmonic decomposition
\begin{equation*}
  \bE_{0}=(\alpha_{0},\beta_{0},\bd'_{0},\bv'_{0},\bH_{0})
\end{equation*}
with
\begin{equation*}
  \alpha_{0}=1531,
  \qquad
  \beta_{0}=1479,
\end{equation*}

\begin{equation*}
  \bd_{0}'=
  \begin{pmatrix}
    \frac{11}{3} & 2           & 14            \\
    2            & \frac{5}{3} & 23            \\
    14           & 23          & -\frac{16}{3}
  \end{pmatrix}
  \text{ GPa , }
  \qquad
  \bv_{0}'=
  \begin{pmatrix}
    -1  & -11 & -1 \\
    -11 & 9   & -1 \\
    -1  & -1  & -8
  \end{pmatrix}
  \text{ GPa},
\end{equation*}

\begin{equation*}
  \bE_{0}^{iso}= \frac{1}{15}\left(\alpha_{0}+2 \beta_{0}\right)  \bq\underset{(4)}{\otimes} \bq+\frac{1}{6}\left(\alpha_{0}-\beta_{0}\right) \bq\underset{(2,2)}{\otimes} \bq,
\end{equation*}
and, in Voigt notation,
\begin{equation*}
  [\bH_{0}]=
  \begin{pmatrix}
    -\frac{1986}{35} & \frac{1093}{35} & \frac{893}{35} & 5 & \frac{352}{7} & -\frac{99}{7} \\ \\ \frac{1093}{35} & -\frac{2306}{35} & \frac{1213}{35} & -31 & \frac{3}{7} & \frac{132}{7} \\ \\ \frac{893}{35} & \frac{1213}{35} & -\frac{2106}{35} & 26 & -\frac{355}{7} & -\frac{33}{7}  \\ \\ 5 & -31 & 26 & \frac{1213}{35} & -\frac{33}{7} & \frac{3}{7} \\  \\ \frac{352}{7} & \frac{3}{7} & -\frac{355}{7} & -\frac{33}{7} & \frac{893}{35} & 5 \\ \\ -\frac{99}{7} & \frac{132}{7} & -\frac{33}{7} & \frac{3}{7} & 5 & \frac{1093}{35}
  \end{pmatrix}
  \text{ GPa}.
\end{equation*}

\section{Components of second-order covariant $\bd_{2}'$}
\label{sec:d2-prime}

\subsection{Elasticity tensor case}

The components $(\bd_{2}')_{ij}=(\bd_{2})_{ij}-\frac{1}{3} (\bd_{2}')_{kk} \delta_{ij}$ of $\bd_{2}=\bH\3dots \bH$, with $\bH$ the fourth-order harmonic tensor given by~\eqref{eq:HVoigt}, are:
\begin{align*}
  (\bd_{2}')_{11} & = \frac{2}{3} \big(-4 \Lambda_{1}^{2}-\Lambda_{1} \Lambda_{2}-\Lambda_{1} \Lambda_{3}+2 \Lambda_{2}^{2}+2 \Lambda_{2} \Lambda_{3}+2 \Lambda_{3}^{2}  +X_{1}^{2}-4 X_{1} X_{2}-4 X_{2}^{2}
  \\
                  & \quad +Y_{1}^{2}+5 Y_{1} Y_{2}+2 Y_{2}^{2}-2 Z_{1}^{2}-Z_{1} Z_{2}+2 Z_{2}^{2}\big),
  \\
  (\bd_{2}')_{22} & = -\frac{2}{3} \big(-2 \Lambda_{1}^{2}+\Lambda_{1} \Lambda_{2}-2 \Lambda_{1} \Lambda_{3}+4 \Lambda_{2}^{2}+\Lambda_{2} \Lambda_{3}-2 \Lambda_{3}^{2} +2 X_{1}^{2}+X_{1} X_{2}-2 X_{2}^{2}
  \\
                  & \quad -Y_{1}^{2}+4 Y_{1} Y_{2}+4 Y_{2}^{2}-Z_{1}^{2}-5 Z_{1} Z_{2}-2 Z_{2}^{2}\big),
  \\
  (\bd_{2}')_{12} & = 3 X_{1} Y_{1} + 3 X_{2} Y_{1} - 4 X_{1} Y_{2} - X_{2} Y_{2} + 4 Z_{1}\Lambda_{1} +
  Z_{2}\Lambda_{1} + 3 Z_{1}\Lambda_{2} - Z_{2}\Lambda_{2} - 2 Z_{1}\Lambda_{3},
  \\
  (\bd_{2}')_{13} & = 3 X_{1} (Z_{1} + Z_{2}) - X_{2} (4 Z_{1} + Z_{2}) + 3 Y_{1}\Lambda_{1} - Y_{2}\Lambda_{1} - 2 Y_{1}\Lambda_{2} +
  4 Y_{1}\Lambda_{3} + Y_{2}\Lambda_{3},
  \\
  (\bd_{2}')_{23} & = 3 Y_{1} Z_{1} + 3 Y_{2} Z_{1} - 4 Y_{1} Z_{2} - Y_{2} Z_{2} - 2 X_{1}\Lambda_{1} + 4 X_{1}\Lambda_{2} + X_{2}\Lambda_{2} + 3 X_{1}\Lambda_{3} - X_{2}\Lambda_{3}.
\end{align*}

\subsection{Piezoelectricity tensor case}

The components $(\bd_{2}')_{ij}=(\bd_{2})_{ij}-\frac{1}{3} (\bd_{2}')_{kk} \delta_{ij}$ of $\bd_{2}=\bh: \bh$ are:
\begin{align*}
  (\bd_{2}')_{11} & = \frac{2}{3} \left(h_{111}^{2}-3 h_{112} H_{222}-2 h_{222}^{2}+3 h_{223} h_{333}+h_{333}^{2}\right),
  \\
  (\bd_{2}')_{22} & = -\frac{2}{3} \left(2 h_{111}^{2}+3 h_{111} h_{122}-h_{222}^{2}+h_{333} (3 h_{223}+2 h_{333})\right),
  \\
  (\bd_{2}')_{12} & = h_{111} (2 h_{112}+h_{222})+3 h_{112} h_{122}+2 h_{122} h_{222}-2 h_{123} h_{333},
  \\
  (\bd_{2}')_{13} & = h_{111} h_{223}+h_{122} (3 h_{223}+h_{333})-2 h_{123} h_{222},
  \\
  (\bd_{2}')_{23} & = -2 h_{111} h_{123}-h_{112} (3 h_{223}+2 h_{333})-h_{222} (h_{223}+h_{333}).
\end{align*}

\section{Proof of theorem~\ref{thm:Pcubic}}
\label{sec:proof-d2-prime}

Smith and Bao~\cite{SB1997} have derived a minimal integrity basis of five invariants for the algebra $\RR[\HH^{3}(\RR^{3})]^{\SO(3)}$, of polynomial $\SO(3)$-invariants of the third-order harmonic tensors $\bh\in \HH^{3}(\RR^{3})$. These five invariants (equations (2.3) and (2.4) in~\cite{SB1997}) can be recast in a more intrinsic form as
\begin{gather*}
  I_{2} = \norm{\bh}^{2}, \qquad  K_{4} = \tr \bd_{2}^{\, 2}, \qquad I_6 =\norm{\vv_{3}}^{2}, \\
  K_{10} = \bh(\vv_{3}, \vv_{3}, \vv_{3}), \qquad K_{15} = \det(\vv_{3}, \bd_{2} \cdot \vv_{3}, \vv_{3}\cdot \bh\cdot \vv_{3}),
\end{gather*}
where  $\bd_{2}=\bh:\bh$ and $\vv_{3} := \bh:\bd_{2}'$.
In~\cite{OA2014}, Olive and Auffray have used these results to deduce that a minimal integrity basis for the algebra $\RR[\HH^{3}(\RR^{3})]^{\OO(3)}$, of polynomial $\OO(3)$-invariants of $\bh\in \HH^{3}(\RR^{3})$ consists of the four invariants
\begin{equation*}
  I_{2}, \quad K_{4}, \quad I_6, \quad \text{and} \quad K_{10}.
\end{equation*}
Here, we will formulate alternative integrity bases for both $\RR[\HH^{3}(\RR^{3})]^{\SO(3)}$ and $\RR[\HH^{3}(\RR^{3})]^{\OO(3)}$, which happen to be more useful in order to characterize the cubic symmetry class in $\HH^{3}(\RR^{3})$ for $\OO(3)$. These will be used to prove theorem~\ref{thm:Pcubic}.

\begin{thm}\label{lem:Ik}
  Let $\bh\in\HH^{3}(\RR^{3})$ be an harmonic third-order tensor, $\bd_{2}=\bh:\bh$, and
  \begin{equation*}
    \vv_{3} := \bh:\bd_{2}', \qquad \vv_5 := \bd_{2}'\cdot \vv_{3}, \qquad \vv_7 := \bd_{2}'\cdot \vv_5 \quad \text{where} \quad \bd_{2}' = \bd_{2}-\frac{1}{3}\tr(\bd_{2})\bq.
  \end{equation*}
  \begin{enumerate}
    \item A minimal integrity basis of $\RR[\HH^{3}(\RR^{3})]^{\SO(3)}$ is constituted by the five invariants
          \begin{gather*}
            I_{2} := \tr \bd_{2}=\norm{\bh}^{2}, \qquad I_{4} := \tr( \bd_{2}^{\prime \,2})= \norm{\bd_{2}^{\prime}}^{2},\qquad
            I_6 := \norm{\vv_{3}}^{2}, \\
            I_{10} := \norm{\bd_{2}' \times \vv_{3}}^{2}, \qquad I_{15} := \det(\vv_{3}, \vv_5, \vv_7).
          \end{gather*}
    \item A minimal integrity basis of $\RR[\HH^{3}(\RR^{3})]^{\OO(3)}$ is constituted by the four invariants $I_{2}$, $I_{4}$, $I_{6}$, and $I_{10}$.
  \end{enumerate}
\end{thm}

\begin{proof}
  To prove the theorem, it is enough to show that Smith and Bao's invariants can be expressed as polynomials of $I_{2}$, $I_{4}$, $I_{6}$, $I_{10}$, $I_{15}$, since, then, this set will be generating and moreover the cardinal of a minimal integrity basis of homogeneous invariants does not depend on the choice of a particular basis~\cite{DL1985/86}. Indeed, one can check that
  \begin{align*}
     & K_{4} = I_{4} + \frac{1}{3} {I_{2}}^{2},
    \\
     & K_{10} = - \frac{4}{3}I_{10} - \frac{1}{27} {I_{2}}^{3} I_{4} + \frac{1}{9} {I_{2}}^{2} I_{6} + \frac{2}{9}I_{2} {I_{4}}^{2} + \frac{2}{3} I_{4} I_{6},
    \\
     & K_{15} = 2 I_{15},
  \end{align*}
  which achieves the proof.
\end{proof}

\begin{proof}[Proof of theorem~\ref{thm:Pcubic}]
  Let $\be=\bg+ \bh$ be a piezoelectricity tensor,
  \begin{equation*}
    \bg=(\vv, \ww, \ba) \in \HH^{1}(\RR^3)\oplus \HH^{1}(\RR^3) \oplus \HH^{2 \sharp}(\RR^3),
    \qquad
    \bh\in \HH^{3}(\RR^3).
  \end{equation*}
  If $\be \in \cstrata{\octa^{-}}$, then $\bg=(\vv, \ww, \ba)\in \cstrata{\octa^{-}}$ vanishes since an element in $ \HH^{1}(\RR^3)$ or $\HH^{2 \sharp}(\RR^3)$ with at least cubic symmetry ($[\octa^{-}]$) is necessarily isotropic. For the same reason $\bd_{2}'(\bh)=0$.
  Conversely, if $\bg=0$, then $\be=\bh\in \HH^{3}(\RR^3)$ is harmonic, and it suffices to show that $\bh\in \cstrata{\octa^{-}}$ (is at least cubic). Since we assume furthermore $\bd_{2}'=0$, we have
  \begin{equation*}
    I_{2}(\bh) = \norm{\bh}^{2} \ge 0, \qquad I_{4}(\bh) = 0, \qquad I_{6}(\bh) = 0, \qquad I_{10}(\bh) = 0.
  \end{equation*}
  Now an harmonic tensor in $\HH^{3}(\RR^{3})$ which is fixed by $\octa^{-}$ is written (in Voigt notation~\eqref{eq:Hpiezomatrice}) as
  \begin{equation*}
    \bh_{0} = \delta
    \begin{pmatrix}
      0 & 0 & 0 & 1 & 0 & 0 \\
      0 & 0 & 0 & 0 & 1 & 0 \\
      0 & 0 & 0 & 0 & 0 & 1 \\
    \end{pmatrix}.
  \end{equation*}
  For such a tensor we get
  \begin{equation*}
    I_{2}(\bh_{0}) = 6 \delta^{2}, \qquad I_{4}(\bh_{0}) = 0, \qquad I_{6}(\bh_{0}) = 0, \qquad I_{10}(\bh_{0}) = 0.
  \end{equation*}
  Therefore, since $I_{2}(\bh)=\norm{\bh}^{2}\geq 0$, we can find a real number $\delta$ such that $6 \delta^{2} = I_{2}(\bh)$, and thus an at least cubic tensor $\bh_{0}$ such that
  \begin{equation*}
    I_{2}(\bh_{0}) = I_{2}(\bh),
    \qquad
    I_{4}(\bh_{0}) = I_{4}(\bh),
    \qquad
    I_{6}(\bh_{0}) = I_{6}(\bh),
    \qquad
    I_{10}(\bh_{0}) = I_{10}(\bh).
  \end{equation*}
  But an integrity basis for a real representation of a compact group \emph{separate the orbits}~\cite[Appendix C]{AS1983}.
  Hence, $\bh$ and $\bh_{0}$ are necessarily in the same orbit, which means that $\bh = \rho_{3}(g) \bh_{0}$, for some $g \in \OO(3)$.
\end{proof}

\section{Raw piezoelectricity tensors for wurtzite}
\label{sec:Piez-Mannna}

The raw piezoelectricity tensors $\be_{0}^{x}$ considered in \autoref{sec:piezo} correspond to the mean values computed in~\cite{MTG2018} for
wurtzite $\text{Cr}_{x}\text{Al}_{1-x}\text{N}$, with $x$ the chromium concentration (in C/m$^{2}$),

\begin{align*}
  [\be_{0}^{0.035}]= &
  \left(
  \begin{array}{cccccc}
      -0.0329 & 0.0599  & -0.0195 & 0.0267  & -0.2327 & -0.0988 \\
      -0.0548 & -0.0129 & -0.0063 & -0.2075 & -0.0051 & -0.0293 \\
      -0.5872 & -0.4900 & 1.5560  & -0.0218 & -0.0278 & -0.0115 \\
    \end{array}
  \right),
  \\
  [\be_{0}^{0.07}]=  &
  \left(
  \begin{array}{cccccc}
      -0.0393 & 0.0185  & 0.0048 & 0.0290  & -0.2171 & -0.0436 \\
      -0.07   & 0.0554  & 0.0137 & -0.1574 & 0.0198  & 0.0044  \\
      -0.5179 & -0.5886 & 1.6521 & -0.0085 & -0.0095 & -0.0119 \\
    \end{array}
  \right),
  \\
  [\be_{0}^{0.10}]=  &
  \left(
  \begin{array}{cccccc}
      0.0291  & -0.0141 & -0.0523 & -0.0016 & 0.0028 & 0.0138  \\
      -0.0611 & 0.0819  & -0.0567 & -0.1841 & 0.0116 & 0.0270  \\
      -0.5244 & -0.5918 & 1.7715  & -0.0018 & 0.0066 & -0.0145 \\
    \end{array}
  \right),
  \\
  [\be_{0}^{0.13}]=  &
  \left(
  \begin{array}{cccccc}
      -0.0985 & 0.1138  & 0.047   & -0.0169 & -0.0169 & -0.0984 \\
      0.0558  & 0.0183  & -0.0367 & -0.1735 & -0.0384 & 0.0474  \\
      -0.5441 & -0.5455 & 1.8506  & -0.0148 & -0.0016 & -0.0193 \\
    \end{array}
  \right),
  \\
  [\be_{0}^{0.16}]=  &
  \left(
  \begin{array}{cccccc}
      0.0315  & -0.0375 & 0.0273 & 0.0206  & 0.0206 & 0.0933  \\
      -0.215  & -0.0717 & 0.0845 & -0.2157 & 0.0438 & -0.0332 \\
      -0.4517 & -0.5587 & 1.9243 & 0.0447  & 0.0277 & -0.0482 \\
    \end{array}
  \right)	,             \\
  [\be_{0}^{0.19}]=  &
  \left(
  \begin{array}{cccccc}
      0.4524  & 0.3564  & -0.0827 & -0.0276 & -0.0276 & 0.1067  \\
      -0.0783 & 0.0868  & 0.0318  & 0.0037  & -0.1053 & -0.0765 \\
      -0.5768 & -0.4566 & 2.0350  & -0.1332 & -0.1016 & -0.1253 \\
    \end{array}
  \right),
  \\
  [\be_{0}^{0.225}]= &
  \left(
  \begin{array}{cccccc}
      0.0428  & 0.0974  & -0.0429 & -0.0319 & -0.0363 & 0.0099  \\
      -0.1399 & -0.2386 & -0.0253 & -0.1505 & 0.0143  & -0.1770 \\
      -0.5800 & -0.5552 & 2.2197  & 0.0164  & 0.0048  & 0.0234  \\
    \end{array}
  \right),
  \\
  [\be_{0}^{0.255}]= &
  \left(
  \begin{array}{cccccc}
      -0.0914 & 0.0758  & 0.0000 & -0.0022 & -0.2835 & 0.0000  \\
      0.0000  & -0.0022 & 0.0000 & -0.2660 & 0.0002  & -0.0020 \\
      -0.6063 & -0.5847 & 2.3709 & -0.0714 & -0.0738 & -0.0559 \\
    \end{array}
  \right).
\end{align*}



\begin{thebibliography}{10}

\bibitem{AS1981}
M.~Abud and G.~Sartori.
\newblock {T}he geometry of orbit-space and natural minima of {H}iggs
  potentials.
\newblock {\em Phys. Lett. B}, 104(2):147--152, 1981.

\bibitem{AS1983}
M.~Abud and G.~Sartori.
\newblock {T}he geometry of spontaneous symmetry breaking.
\newblock {\em Ann. Physics}, 150(2):307--372, 1983.

\bibitem{Ali1995}
F.~Alizadeh.
\newblock Interior point methods in semidefinite programming with applications
  to combinatorial optimization.
\newblock {\em SIAM Journal on Optimization}, 5(1):13--51, 1995.

\bibitem{ADKD2021}
A.~Antonelli, B.~Desmorat, B.~Kolev, and R.~Desmorat.
\newblock Distance to plane elasticity orthotropy by {E}uler-{L}agrange method,
  arxiv, doi: 10.48550/arxiv.2107.14456, 2021.

\bibitem{Art1993}
R.~Arts.
\newblock {\em A study of general anisotropic elasticity in rocks by wave
  propagation}.
\newblock PhD thesis, PhD University Pierre et Marie Curie, Paris 6, 1993.

\bibitem{AKP2014}
N.~Auffray, B.~Kolev, and M.~Petitot.
\newblock On anisotropic polynomial relations for the elasticity tensor.
\newblock {\em J. Elasticity}, 115(1):77--103, 2014.

\bibitem{Bac1970}
G.~Backus.
\newblock A geometrical picture of anisotropic elastic tensors.
\newblock {\em Reviews of geophysics}, 8(3):633--671, 1970.

\bibitem{BCM2007}
D.~Benterki, J.-P. Crouzeix, and B.~Merikhi.
\newblock A numerical feasible interior point method for linear semidefinite
  programs.
\newblock {\em RAIRO Oper. Res.}, 41(1):49--59, 2007.

\bibitem{BCR2013}
J.~Bochnak, M.~Coste, and M.-F. Roy.
\newblock {\em Real Algebraic Geometry}, volume~36.
\newblock Springer Berlin Heidelberg, Nov. 2013.

\bibitem{Bre1960}
G.~E. Bredon.
\newblock Finiteness of number of orbit types.
\newblock In {\em A. Borel, Seminar on transformation groups. With
  contributions by G. Bredon, EE Floyd, D. Montgomery, R. Palais. Annals of
  Mathematics Studies}, number~46, 1960.

\bibitem{Bre1972}
G.~E. Bredon.
\newblock {\em Introduction to compact transformation groups}.
\newblock Academic press, 1972.

\bibitem{BS2008}
I.~Bucataru and M.~A. Slawinski.
\newblock Invariant properties for finding distance in space of~elasticity
  tensors.
\newblock {\em Journal of Elasticity}, 94(2):97--114, nov 2008.

\bibitem{Cia1988}
P.~G. Ciarlet.
\newblock {\em Mathematical elasticity. Vol. I}, volume~20 of {\em Studies in
  Mathematics and its Applications}.
\newblock North-Holland Publishing Co., Amsterdam, 1988.
\newblock Three-dimensional elasticity.

\bibitem{Cos2002}
M.~Coste.
\newblock {A}n {I}ntroduction to {S}emialgebraic {G}eometry.
\newblock Université de Rennes 1, 2002.

\bibitem{CF1996}
R.~Curto and L.~Fialkow.
\newblock Solution of the truncated complex moment problem for flat data.
\newblock {\em Memoirs of the American Mathematical Society},
  119(568):2825--2855, 1996.

\bibitem{CF2000}
R.~Curto and L.~Fialkow.
\newblock The truncated complex $k$-moment problem.
\newblock {\em Transactions of the American Mathematical Society},
  352(10):2825--2855, 2000.

\bibitem{Dan1963}
G.~Dantzig.
\newblock Linear programming and extensions, princeton, univ.
\newblock {\em Press, Princeton, NJ}, 1963.

\bibitem{Del2005}
J.~Dellinger.
\newblock Computing the optimal transversely isotropic approximation of a
  general elastic tensor.
\newblock {\em {Geophysics}}, 70(5):11--20, 2005.

\bibitem{DHer2012}
D.~Den~Hertog.
\newblock {\em Interior point approach to linear, quadratic and convex
  programming: algorithms and complexity}, volume 277.
\newblock Springer Science \& Business Media, 2012.

\bibitem{DL1985/86}
J.~Dixmier and D.~Lazard.
\newblock Le nombre minimum d'invariants fondamentaux pour les formes binaires
  de degré {$7$}.
\newblock {\em Portugal. Math.}, 43(3):377--392, 1985/86.

\bibitem{EM1990}
A.~Eringen and G.~Maugin.
\newblock {\em Electrodynamics of Continua , tomes I et II}.
\newblock Springer-Verlag, 1990.

\bibitem{FV1996}
S.~Forte and M.~Vianello.
\newblock Symmetry classes for elasticity tensors.
\newblock {\em Journal of Elasticity}, 43(2):81--108, 1996.

\bibitem{FV1997}
S.~Forte and M.~Vianello.
\newblock {S}ymmetry classes and harmonic decomposition for photoelasticity
  tensors.
\newblock {\em International Journal of Engineering Science},
  35(14):1317--1326, 1997.

\bibitem{FBG1996}
M.~Fran{ç}ois, Y.~Berthaud, and G.~Geymonat.
\newblock Une nouvelle analyse des symétries d'un matériau élastique
  anisotrope. exemple d'utilisation à partir de mesures ultrasonores.
\newblock {\em C. R. Acad. Sci. Paris, Série IIb}, 322:87--94, 1996.

\bibitem{FGB1998}
M.~Fran{ç}ois, G.~Geymonat, and Y.~Berthaud.
\newblock Determination of the symmetries of an experimentally determined
  stiffness tensor: Application to acoustic measurements.
\newblock {\em Int. J. Sol. Struct.}, 35(31-32):4091--4106, 1998.

\bibitem{Fre2004}
R.~M. Freund.
\newblock Introduction to semidefinite programming (sdp).
\newblock {\em Massachusetts Institute of Technology}, pages 8--11, 2004.

\bibitem{GTT1963}
D.~Gazis, I.~Tadjbakhsh, and R.~Toupin.
\newblock The elastic tensor of given symmetry nearest to an anisotropic
  elastic tensor.
\newblock {\em Acta Crystallographica}, 16(9):917--922, 1963.

\bibitem{GW2002}
G.~Geymonat and T.~Weller.
\newblock Symmetry classes of piezoelectric solids.
\newblock {\em Comptes rendus de l'Acad\'emie des Sciences. S\'erie I},
  335:847--8524, 2002.

\bibitem{GSS1988}
M.~Golubitsky, I.~Stewart, and D.~G. Schaeffer.
\newblock {\em Singularities and groups in bifurcation theory. {V}ol. {II}},
  volume~69 of {\em Applied Mathematical Sciences}.
\newblock Springer-Verlag, New York, 1988.

\bibitem{GLS1981}
M.~Grötschel, L.~Lovász, and A.~Schrijver.
\newblock The ellipsoid method and its consequences in combinatorial
  optimization.
\newblock {\em Combinatorica}, 1(2):169--197, 1981.

\bibitem{Hav1935}
E.~K. Haviland.
\newblock On the momentum problem for distribution functions in more than one
  dimension.
\newblock {\em American Journal of Mathematics}, 57(3):562--568, 1935.

\bibitem{HL2003}
D.~Henrion and J.-B. Lasserre.
\newblock Glopti{P}oly: global optimization over polynomials with {M}atlab and
  {S}e{D}u{M}i.
\newblock {\em ACM Trans. Math. Software}, 29(2):165--194, 2003.

\bibitem{HLL2009}
D.~Henrion, J.-B. Lasserre, and J.~Löfberg.
\newblock Glopti{P}oly 3: moments, optimization and semidefinite programming.
\newblock {\em Optim. Methods Softw.}, 24(4-5):761--779, 2009.

\bibitem{JLR2008}
M.~L. J.-B.~Lasserre and P.~Rostalski.
\newblock Semidefinite characterization and computation of real radical ideals.
\newblock {\em Foundations of Computational Mathematics}, (8):607--647, 2008.

\bibitem{JP2001}
T.~Jacobi and A.~Prestel.
\newblock Distinguished representations of strictly positive polynomials.
\newblock {\em Journal für die reine und angewandte Mathematik},
  2001(532):223--235, jan 2001.

\bibitem{Jar1993}
F.~Jarre.
\newblock An interior-point method for minimizing the maximum eigenvalue of a
  linear combination of matrices.
\newblock {\em SIAM Journal on Control and Optimization}, 31(5):1360--1377,
  1993.

\bibitem{JLL2014}
V.~Jeyakumar, J.-B. Lasserre, and G.~Li.
\newblock On polynomial optimization over non-compact semi-algebraic sets.
\newblock {\em J. Optim. Theory Appl.}, 163(3):707--718, 2014.

\bibitem{KS2008}
M.~Kochetov and M.~A. Slawinski.
\newblock On obtaining effective transversely isotropic elasticity~tensors.
\newblock {\em Journal of Elasticity}, 94(1):1--13, oct 2008.

\bibitem{KS2009}
M.~Kochetov and M.~A. Slawinski.
\newblock On obtaining effective orthotropic elasticity tensors.
\newblock {\em The Quarterly Journal of Mechanics and Applied Mathematics},
  62(2):149--166, mar 2009.

\bibitem{Las2001}
J.-B. Lasserre.
\newblock Global optimization with polynomials and the problem of moments.
\newblock {\em SIAM J. Optim.}, 11(3):796--817, Jan. 2001.

\bibitem{Las2009}
J.-B. Lasserre.
\newblock {\em Moments, Positive Polynomials and Their Applications}.
\newblock Imperial College Press, oct 2009.

\bibitem{Las2015}
J.-B. Lasserre.
\newblock {\em An Introduction to Polynomial and Semi-Algebraic Optimization}.
\newblock Cambridge University Press, 2015.

\bibitem{Lau2005}
M.~Laurent.
\newblock Revisiting two theorems of \mbox{Curto} and \mbox{Fialkow} on moment
  matrices.
\newblock {\em Proceedings of the American Mathematical Society},
  133(10):2965--2976, 2005.

\bibitem{Lau2009}
M.~Laurent.
\newblock Sums of squares, moment matrices and optimization over polynomials.
\newblock In {\em Emerging applications of algebraic geometry}, volume 149 of
  {\em IMA Vol. Math. Appl.}, pages 157--270. Springer, New York, 2009.

\bibitem{LC1985}
J.~Lemaitre and J.-L. Chaboche.
\newblock {\em M\'ecanique des mat\'eriaux solides}.
\newblock Dunod, english translation 1990 'Mechanics of Solid Materials'
  Cambridge University Press, 1985.

\bibitem{Lof2004}
J.~Lofberg.
\newblock Yalmip: A toolbox for modeling and optimization in matlab.
\newblock pages 284--289, 2004.

\bibitem{LY2008}
D.~G. Luenberger and Y.~Ye.
\newblock Linear and nonlinear programming, vol. 116, 2008.

\bibitem{Man1962}
L.~N. Mann.
\newblock Finite orbit structure on locally compact manifolds.
\newblock {\em Michigan Mathematical Journal}, 9(1):87--92, Jan. 1962.

\bibitem{MTG2018}
S.~Manna, K.~R. Talley, P.~Gorai, J.~Mangum, A.~Zakutayev, G.~L. Brennecka,
  V.~Stevanovi{\'{c}}, and C.~V. Ciobanu.
\newblock Enhanced piezoelectric response of {AlN} via {CrN} alloying.
\newblock {\em Physical Review Applied}, 9(3):034026, mar 2018.

\bibitem{Mar2008}
M.~Marshall.
\newblock {\em Positive polynomials and sums of squares}.
\newblock Number 146. American Mathematical Soc., 2008.

\bibitem{MDC2019}
A.~Mattiello, R.~Desmorat, and J.~Cormier.
\newblock Rate dependent ductility and damage threshold: Application to
  nickel-based single crystal {CMSX}-4.
\newblock {\em International Journal of Plasticity}, 113:74--98, feb 2019.

\bibitem{MMRFS2008}
P.~Mayrhofer, D.~Music, T.~Reeswinkel, H.-G. Fuß, and J.~Schneider.
\newblock Structure, elastic properties and phase stability of
  {C}r$_{1-x}${A}l$_{x}$n.
\newblock {\em Acta Materialia}, 56(11):2469--2475, 2008.

\bibitem{Mev2007}
M.~Mevissen.
\newblock Introduction to concepts and advances in polynomial optimization.
\newblock {\em Review available at https://inf. ethz.
  ch/personal/fukudak/semi/optpast/FS07/opt\_abs/PolynomialOptimization. pdf,
  Institute for Operations Research, ETH, Zurich}, 2007.

\bibitem{MN2006}
M.~Moakher and A.~N. Norris.
\newblock The closest elastic tensor of arbitrary symmetry to an elasticity
  tensor of lower symmetry.
\newblock {\em Journal of Elasticity}, 85(3):215--263, 2006.

\bibitem{Mos1957}
G.~D. Mostow.
\newblock On a conjecture of montgomery.
\newblock {\em Annals of Mathematics}, 65(3):513--516, 1957.

\bibitem{NN1994}
Y.~Nesterov and A.~Nemirovskii.
\newblock {\em Interior-point polynomial algorithms in convex programming}.
\newblock SIAM, 1994.

\bibitem{NT1997}
Y.~E. Nesterov and M.~J. Todd.
\newblock Self-scaled barriers and interior-point methods for convex
  programming.
\newblock {\em Mathematics of Operations Research}, 22(1):1--42, 1997.

\bibitem{Nye1985}
J.~F. Nye.
\newblock {\em Physical Properties of Crystals}.
\newblock Oxford University Press, May 1985.

\bibitem{OA2014}
M.~Olive and N.~Auffray.
\newblock Isotropic invariants of a completely symmetric third-order tensor.
\newblock {\em Journal of Mathematical Physics}, 55(9):092901, sep 2014.

\bibitem{OA2021}
M.~Olive and N.~Auffray.
\newblock Symmetry classes in piezoelectricity from second-order symmetries.
\newblock {\em Mathematics and Mechanics of Complex Systems}, 9(1):77--105, mar
  2021.

\bibitem{OKDD2021}
M.~Olive, B.~Kolev, R.~Desmorat, and B.~Desmorat.
\newblock Characterization of the symmetry class of an elasticity tensor using
  polynomial covariants.
\newblock {\em Mathematics and Mechanics of Solids}, 27(1):144--190, may 2021.

\bibitem{OV1990}
A.~Onishchik and E.~Vinberg.
\newblock {\em Lie Groups and Algebraic Groups}.
\newblock Springer-Verlag, Berlin Heidelberg, 1990.

\bibitem{PS1985}
C.~Procesi and G.~Schwarz.
\newblock {I}nequalities defining orbit spaces.
\newblock {\em Invent. Math.}, 81(3):539--554, 1985.

\bibitem{Put1993}
M.~Putinar.
\newblock Positive polynomials on compact semi-algebraic sets.
\newblock {\em Indiana Univ. Math. J.}, 42(3):969--984, 1993.

\bibitem{Sch1989}
G.~W. Schwarz.
\newblock The topology of algebraic quotients.
\newblock In {\em Topological methods in algebraic transformation groups},
  pages 135--151. Springer, 1989.

\bibitem{Sch2005}
M.~Schweighofer.
\newblock Optimization of polynomials on compact semialgebraic sets.
\newblock {\em SIAM J. Optim.}, 15(3):805--825, 2005.

\bibitem{SB1997}
G.~Smith and G.~Bao.
\newblock Isotropic invariants of traceless symmetric tensors of orders three
  and four.
\newblock {\em International Journal of Engineering Science},
  35(15):1457--1462, dec 1997.

\bibitem{Spe1970}
A.~Spencer.
\newblock A note on the decomposition of tensors into traceless symmetric
  tensors.
\newblock {\em Int. J. Engng Sci.}, 8:475--481, 1970.

\bibitem{Stu1997}
J.~F. Sturm.
\newblock Primal-dual interior point approach to semidefinite programming,
  1997.

\bibitem{Stu1999}
J.~F. Sturm.
\newblock Using {SeDuMi} 1.02, a matlab toolbox for optimization over symmetric
  cones.
\newblock {\em Optimization Methods and Software}, 11(1-4):625--653, jan 1999.

\bibitem{Tod2001}
M.~J. Todd.
\newblock Semidefinite optimization.
\newblock {\em Acta Numerica}, 10:515--560, 2001.

\bibitem{VB1996}
L.~Vandenberghe and S.~Boyd.
\newblock Semidefinite programming.
\newblock {\em SIAM review}, 38(1):49--95, 1996.

\bibitem{Van2020}
R.~J. Vanderbei.
\newblock {\em Linear programming: foundations and extensions}, volume 285.
\newblock Springer Nature, 2020.

\bibitem{Via1997}
M.~Vianello.
\newblock An integrity basis for plane elasticity tensors.
\newblock {\em Arch. Mech. (Arch. Mech. Stos.)}, 49(1):197--208, 1997.

\bibitem{WSV2012}
H.~Wolkowicz, R.~Saigal, and L.~Vandenberghe.
\newblock {\em Handbook of Semidefinite Programming}, volume~27.
\newblock Springer US, Dec. 2012.

\bibitem{ZB1994}
Q.~S. Zheng and J.~P. Boehler.
\newblock The description, classification, and reality of material and physical
  symmetries.
\newblock {\em Acta Mechanica}, 102(1-4):73--89, mar 1994.

\bibitem{ZTP2013}
W.-N. Zou, C.-X. Tang, and E.~Pan.
\newblock Symmetry types of the piezoelectric tensor and their identification.
\newblock {\em Proceedings of the Royal Society A: Mathematical, Physical and
  Engineering Sciences}, 469(2155):20120755, 2013.

\end{thebibliography}
\end{document}